\documentclass[a4paper,12pt,reqno]{amsart}
\usepackage{amsfonts}
\usepackage{amsmath,array}
\usepackage{amssymb}
\usepackage{mathrsfs, mathtools}
\usepackage[colorlinks]{hyperref}
 \usepackage{graphicx}
\usepackage{enumerate}
\usepackage[usenames]{color}
\usepackage{comment}

\makeatletter
\newcommand{\setword}[2]{%
  \phantomsection
  #1\def\@currentlabel{\unexpanded{#1}}\label{#2}%
}
\makeatother
\newtheorem{thm}{Theorem}[section]
\newtheorem{cor}[thm]{Corollary}
\newtheorem{lem}[thm]{Lemma}
\newtheorem{prop}[thm]{Proposition}

\newtheorem{remark}[thm]{Remark}
\numberwithin{equation}{section}
\usepackage[english]{babel} 
\usepackage{blindtext}

\theoremstyle{definition}
\newtheorem{definition}[thm]{Definition}

\usepackage[colorlinks]{hyperref}
\usepackage{cite}

\begin{document}

\allowdisplaybreaks 

 \title[Nontrivial solutions for nonlinear problems]{Nontrivial solutions for nonlinear problems driven by a superposition of fractional $p$-Laplacians with Neumann boundary conditions}

 \author[Yergen Aikyn]{Yergen Aikyn}

\address[Yergen Aikyn]{Department of Mathematics: Analysis, Logic and Discrete Mathematics, Ghent University, Ghent, Belgium}
\email{yergen.aikyn@ugent.be}

\date{}

\begin{abstract}
In this paper, we investigate existence results for nonlinear nonlocal problems governed by an operator obtained as a superposition of fractional $p$-Laplacians, subject to Neumann boundary conditions. A spectral analysis of the main operator leads us to apply different variational tools to establish our results. Specifically, we employ either the mountain pass method or the technique of linking over cones. Due to the generality of the setting, the resulting theory applies to a broad class of local–nonlocal models.
\end{abstract}


\maketitle

\tableofcontents

\section{Introduction and main results}

This paper studies Neumann boundary value problems driven by a nonlocal operator defined as the combination of the classical $p$-Laplacian and a continuous superposition of fractional $p$-Laplacians.
More precisely, let $\mu$ be a nonnegative and nontrivial finite Borel measure over $(0,1)$, and let $\alpha \geq 0$ and $p\in(1,+\infty)$. The operator of interest is defined by
\begin{align}\label{main operator}
\mathfrak{L}_{\alpha, \mu, p}(u):=-\alpha \Delta_p u+\int_{(0,1)}(-\Delta)_{p}^{s} u d \mu(s).
\end{align}
Here, $\Delta_p$ is the classical $p$-Laplacian, given by $\Delta_p u:= \operatorname{div}\left(|\nabla u|^{p-2} \nabla u\right)$, and $(-\Delta)_p^s$ is the fractional $p$-Laplacian, defined for all $s \in(0,1)$ as
\begin{align}\label{p frac laplacian}
(-\Delta)_{p}^{s} u(x)&:=c_{N, s, p} \lim _{\varepsilon \searrow 0} \int_{\mathbb{R}^N \backslash B_{\varepsilon}(x)} \frac{|u(x)-u(y)|^{p-2}(u(x)-u(y))}{|x-y|^{N+s p}} d y.
\end{align}
For further details on the fractional $p$-Laplace operator, see, e.g., \cite{IMS-2016, BLS-2018} and the references therein.
The positive constant $c_{N, s, p}$ is the standard normalization (see \cite[Definition 1.1]{Warma-2016}) chosen so as to ensure consistent limits as $s \nearrow 1$ and $s \searrow 0$, namely
\begin{equation*}
    \lim _{s \nearrow 1}(-\Delta)_{p}^{s} u=(-\Delta)_p^{1} u=-\Delta_p u \quad \text { and } \quad \lim _{s \searrow 0}(-\Delta)_{p}^{s} u=(-\Delta)_{p}^{0} u=u.
\end{equation*}

A continuous superposition of linear fractional operators corresponding to the case $p=2$, i.e., the operator $\mathfrak{L}_{\alpha, \mu, 2}:=\mathfrak{L}_{\alpha, \mu}$, was recently introduced in \cite{DLSV-2025}. In this paper, we let $p \in(1,+\infty)$, which endows the operator $\mathfrak{L}_{\alpha, \mu, p}$ with a nonlinear structure that requires a different analytical approach from the linear case.

More generally, when $\mu$ is a signed finite Borel measure on $[0,1]$, continuous superpositions of operators of different fractional orders have recently been investigated in \cite{AGKR, AGKR1, DPSV25p, DPSV25, DPSV25j}, in the context of Dirichlet boundary problems with critical nonlinearity.

To properly define the Neumann boundary conditions for the operator $\mathfrak{L}_{\alpha, \mu, p}$, one has to take into account both the local and nonlocal contributions in the definition of $\mathfrak{L}_{\alpha, \mu, p}$. In the spirit of \cite{DLSV-2025}, we formulate the Neumann boundary conditions as follows:
\begin{definition}\label{boundary cond}
    We say that $u$ satisfies the ($\alpha, \mu, p$)-Neumann conditions if
\begin{equation}\label{boundary cond 0}
    \begin{cases}\int_{(0,1)} \mathscr{N}_{s,p} u(x) d \mu(s)=g \text { for all } x \in \mathbb{R}^{N} \backslash \bar{\Omega}, & \text { if } \alpha=0,  \\ \left|\nabla u\right|^{p-2}\partial_{\nu} u(x)=h \text { for all } x \in \partial \Omega, & \text { if } \mu \equiv 0, \\ \begin{cases}\left|\nabla u\right|^{p-2}\partial_{\nu} u(x)=h \text { for all } x \in \partial \Omega \text { and } & \\ \int_{(0,1)} \mathscr{N}_{s,p} u(x) d \mu(s)=g \text { for all } x \in \mathbb{R}^{N} \backslash \bar{\Omega}, \end{cases} & \text { if } \alpha \neq 0 \text { and } \mu \not \equiv 0 .\end{cases}
\end{equation}
Moreover, when $g \equiv 0$ in $\mathbb{R}^{N} \backslash \bar{\Omega}$ and $h \equiv 0$ on $\partial \Omega$, we say that \eqref{boundary cond  0} are homogeneous ($\alpha, \mu, p$)-Neumann conditions.
\end{definition} 

The term $\mathscr{N}_{s, p}$ appearing in Definition \ref{boundary cond} is the nonlocal normal $p$-derivative defined by
\begin{equation}\label{Neumann def}
    \mathscr{N}_{s, p} u(x):=c_{N, s, p} \int_{\Omega} \frac{|u(x)-u(y)|^{p-2}(u(x)-u(y))}{|x-y|^{N+p s}} d y, \quad x \in \mathbb{R}^N \backslash \bar{\Omega}.
\end{equation}
This nonlocal version of the normal derivative was originally proposed for $p=2$ in \cite{DRV-2017} and was subsequently extended in \cite{BMPS-2020}. 

For $p=2$, Definition \ref{boundary cond} recovers the Neumann conditions introduced in \cite{DLSV-2025}. Furthermore, when $\mu$ is a Dirac measure concentrated at some $s\in(0,1)$, our formulation is consistent with the nonlocal Neumann conditions given in \cite{ML-2025}.

When $\alpha = 0$, the superposition nature of the operator and the nontriviality of the measure $\mu$ imply the existence of a specific exponent $\widetilde{s}_{\sharp} \in(0,1)$ such that
\begin{equation}\label{exp s sharp}
    \mu([\widetilde{s}_{\sharp}, 1))>0 .
\end{equation}
Moreover, we set
\begin{equation}\label{exp s sharp gen}
    s_{\sharp}:= \begin{cases}\widetilde{s}_{\sharp} & \text { if } \alpha=0,\\ 1 & \text { if } \alpha \neq 0 .\end{cases}
\end{equation}
As established in Corollary \ref{embed full} below, the exponent 
$$p_{s_{\sharp}}^*:=\frac{pN}{N-ps_{\sharp}}$$ 
plays the role of a critical exponent.

Within this framework, we study the existence of solutions to the following problem:
\begin{equation}\label{main problem 1}
\left\{\begin{array}{l}
\mathfrak{L}_{\alpha, \mu, p}(u)+|u|^{p-2}u=\lambda |u|^{p-2}u+f(x, u) \quad \text { in } \Omega,  \\
\text { with homogeneous }(\alpha, \mu, p) \text {-Neumann conditions. }
\end{array}\right.
\end{equation}
Here $\Omega$ is a bounded domain in $\mathbb{R}^N$ with a Lipschitz boundary and the function $f: \Omega \times \mathbb{R} \rightarrow \mathbb{R}$ is a Carathéodory function, that is the map $x \mapsto f(x, t)$ is measurable for every $t \in \mathbb{R}$ and the map $t \mapsto f(x, t)$ is continuous for a.e. $x \in \Omega$. Denoting $F(x, t):=\int_{0}^{t} f(x, \tau) d \tau$, we assume the following hypotheses on $f$:

\setword{($f_1$)}{Word:f1} there exist constants $a_{1}, a_{2}>0$ and $q\in(p,p_{ s_\sharp}^*)$ such that for every $t \in \mathbb{R}$ and for a.e. $x \in \Omega,$
\begin{equation*}
    |f(x, t)| \leq a_{1}+a_{2}|t|^{q-1};
\end{equation*} 

\setword{($f_2$)}{Word:f2} $f(x, t)=o\left(|t|^{p-1}\right)$ as $t \rightarrow 0$ uniformly a.e. in $\Omega$;

\setword{($f_3$)}{Word:f3} there exist $\gamma>p$ and $R \geq 0$ such that for every $t$ with $|t|>R$ and for a.e. $x \in \Omega,$
\begin{equation*}
    0<\gamma F(x, t) \leq f(x, t) t;
\end{equation*}

\setword{($f_4$)}{Word:f4} there exist $\widetilde{\gamma}>p,$ $a_{3}>0$ and $a_{4} \in L^{1}(\Omega)$ such that for every $t \in \mathbb{R}$ and a.e. $x \in \Omega$,
\begin{equation*}
F(x, t) \geq a_{3}|t|^{\widetilde{\gamma}}-a_{4}(x); 
\end{equation*}

\setword{($f_5$)}{Word:f5} if $R>0$ (for $R$ is as in \ref{Word:f3}), then $F(x, t) \geq 0$ for every $t \in \mathbb{R}$ and a.e. $x \in \Omega$.

Condition \ref{Word:f4} was introduced in \cite{Mugnai-0412} to complete the Ambrosetti-Rabinowitz condition for Carathéodory functions. Moreover, by \ref{Word:f1} and \ref{Word:f2} it follows that for any $\varepsilon>0$ there exists $\delta=\delta(\varepsilon)$ such that
\begin{equation}\label{estim for f ex}
        |F(x, t)| \leq \frac{\varepsilon}{p}|t|^{p}+\delta(\varepsilon)|t|^{q} \quad \text{for a.e. } x \in \Omega, \; t \in \mathbb{R}.
    \end{equation}

To carry out our existence analysis, we employ the spectral theory associated with the operator $\mathfrak{L}_{\alpha, \mu, p}+I$, where $I(u):=|u|^{p-2}u$. Namely, we consider the following nonlinear eigenvalue problem:
\begin{equation}\label{Eig problem 0}
\left\{\begin{array}{l}
\mathfrak{L}_{\alpha, \mu, p}(u)+|u|^{p-2}u=\lambda |u|^{p-2}u \quad \text { in } \Omega,  \\
\text { with homogeneous }(\alpha, \mu, p) \text {-Neumann conditions,}
\end{array}\right.
\end{equation}
which consists of finding the values of $\lambda$ that allow for nontrivial solutions. Applying the cohomological index theory of Fadell and Rabinowitz yields a sequence $\{\lambda_k\}$ of eigenvalues satisfying
\begin{equation*}
    1=\lambda_1\leq \lambda_2 \leq \lambda_3 ...
\end{equation*}
Namely, we have the following result.
\begin{thm}\label{Main result 1}
    For all $k \in \mathbb{N},$ the quantity $\lambda_{k}$, defined by 
   \begin{equation*}
\lambda_{k}=\inf _{A \in \mathcal{F}_{k}} \sup _{u \in A} \mathcal{P}(u),
\end{equation*}
    is an eigenvalue of \eqref{Eig problem 0}. Moreover, $\lambda_{k} \rightarrow \infty$, as $k\to \infty$.
\end{thm}
We note that Theorem \ref{Main result 1} is new in this setting. The precise definitions of the sets $A$, $\mathcal{F}_k$, and the functional $\mathcal{P}$ appearing in Theorem \ref{Main result 1} will be given in Section 3.

We prove the existence of nontrivial solutions to problem \eqref{main problem 1}, depending on the position of the parameter $\lambda$ with respect to the first eigenvalue of the eigenvalue problem \eqref{Eig problem 0}. The main results are stated as follows.
\begin{thm}\label{main result 02}
   Let $f: \Omega \times \mathbb{R} \rightarrow \mathbb{R}$ satisfy \ref{Word:f1}-\ref{Word:f5}.
Then, for any $\lambda \geq \lambda_{1}$, problem \eqref{main problem 1} admits a nontrivial weak solution of linking type.
\end{thm}

\begin{thm}\label{main result 03}
Let $f: \Omega \times \mathbb{R} \rightarrow \mathbb{R}$ satisfy \ref{Word:f1}-\ref{Word:f4}.
Then, for any $\lambda < \lambda_{1}$, problem \eqref{main problem 1} admits a nontrivial weak solution of mountain pass type.
\end{thm}

We conclude with some remarks concerning Theorems \ref{main result 02} and \ref{main result 03}. In particular, our existence results extend those established in \cite[Theorems 1.2 and 1.3]{DLSV-2025 Neumann}.

When $\lambda<\lambda_1$, the energy functional associated with problem \eqref{main problem 1} possesses a mountain-pass geometry, yielding a nontrivial solution via the mountain-pass lemma \cite{AR-1973}.

 In the case $\lambda \geq \lambda_{1}$ and $p=2$, the existence result relies on a classical linking argument, which requires a decomposition of the functional space into a direct sum of closed subspaces. In \cite{DLSV-2025 Neumann}, the authors restricted the analysis to a suitable subspace and obtained a complete orthogonal system of eigenfunctions, which then provided the desired direct-sum decomposition. 

However, if $p \neq 2$, the situation changes drastically, since the operator $\mathfrak{L}_{\alpha, \mu, p}$ is no longer linear and the structure of the eigenvalue set of $\mathfrak{L}_{\alpha, \mu, p}$ is not well understood. In this case, we employ the method of linking over cones introduced in \cite{DL-2007}, which provides a suitable alternative framework; see, for instance, 
\cite{ML-2021l, LZ-2011, LMZ-2025, PYZ-2018}. All the details of this theory will be recalled in Section 2.

Because of the generality of the measure $\mu$, problem \eqref{main problem 1} covers many specific cases. We present several examples below:

\begin{itemize}
    \item if $\alpha\neq 0$, $\beta>0$ and $\mu:=\beta \delta_{s}$, being $\delta_{s}$ the Dirac's delta at $s \in(0,1)$, the operator $\mathfrak{L}_{\alpha, \mu, p}$ introduced in \eqref{main operator} boils down to the mixed operator
\begin{equation*}
    \mathfrak{L}_{\alpha, \mu, p}(u)=-\alpha \Delta_p+\beta(-\Delta)_p^{s}
\end{equation*}
for some $s \in(0,1)$, and the Neumann conditions in \eqref{boundary cond  0} reduce to
\begin{equation*}
    \begin{cases}|\nabla u|^{p-2}\partial_{\nu} u(x)=h(x) & \text { for all } x \in \partial \Omega,  \\ \mathscr{N}_{s, p} u(x)=g(x) & \text { for all } x \in \mathbb{R}^{N} \backslash \bar{\Omega}.\end{cases}
\end{equation*}
\end{itemize}
This particular case is considered, for example, in \cite{ML-2025}.
\begin{itemize}
    \item if $\alpha\neq 0$ and the measure given by
$\mu:=\sum_{k=1}^{n} \delta_{s_{k}}$
for some $n \in \mathbb{N}$ with $n \geq 2$, then we have 
\begin{equation*}
    \mathfrak{L}_{\alpha, \mu, p}(u)=-\alpha \Delta_p u+\sum_{k=1}^{n}(-\Delta)_p^{s_{k}} u
\end{equation*}
for $s_{1}, \ldots, s_{n} \in(0,1)$, with $n \geq 2$, with Neumann boundary conditions
\begin{equation*}
    \begin{cases}|\nabla u|^{p-2}\partial_{\nu} u(x)=h(x) & \text { for all } x \in \partial \Omega, \\ \sum_{k=1}^{n} \mathscr{N}_{s_{k}, p} u(x)=g(x) & \text { for all } x \in \mathbb{R}^{N} \backslash \bar{\Omega}.\end{cases}
\end{equation*}
\end{itemize}
    
\begin{itemize}
    \item if $\alpha \neq 0$ and
$\mu:=\sum_{k=1}^{+\infty} c_{k} \delta_{s_{k}}$ (provided that the series converges) with $c_{k} \geq 0$ for any $k \in \mathbb{N} \backslash\{0\}$, then the operator in \eqref{main operator} takes the form
\begin{equation*}
  \mathfrak{L}_{\alpha, \mu, p}(u)=-\alpha \Delta_p u+\sum_{k=1}^{+\infty} c_{k}(-\Delta)_p^{s_{k}} u. 
\end{equation*}
In this case, the Neumann conditions reduce to
\begin{equation*}
    \begin{cases}|\nabla u|^{p-2}\partial_{\nu} u(x)=h(x) & \text { for all } x \in \partial \Omega, \\ \sum_{k=1}^{+\infty} c_{k} \mathscr{N}_{s_{k},p} u(x)=g(x) & \text { for all } x \in \mathbb{R}^{N} \backslash \bar{\Omega}.\end{cases}
\end{equation*}
\end{itemize}

\begin{itemize}
    \item if $\alpha\neq 0$ and the measure $\mu$ given by
\begin{equation*}
d \mu(s):=f(s) d s,
\end{equation*}
where $f$ is a measurable, nonnegative, and not identically zero function and $d s$ is a Lebesgue measure, then the operator in \eqref{main operator} boils down to
\begin{equation*}
    \mathfrak{L}_{\alpha, \mu, p}(u)=-\alpha \Delta_p u+\int_{0}^{1} f(s)(-\Delta)_p^{s} u d s
\end{equation*}
with the Neumann conditions given by
\begin{equation*}
    \begin{cases}|\nabla u|^{p-2}\partial_{\nu} u(x)=h(x) & \text { for all } x \in \partial \Omega, \\ \int_{0}^{1} f(s) \mathscr{N}_{s,p} u(x) d s=g(x) & \text { for all } x \in \mathbb{R}^{N} \backslash \bar{\Omega}.\end{cases}
\end{equation*}
\end{itemize}

The paper is organized as follows. In Section 2, we introduce the functional setting and state several preliminary results, including embedding theorems, integration by parts formulas, and some notions of linking sets and cohomological index theory. In Section 3, we prove Theorem \ref{Main result 1}, and in Section 4, we prove the main existence results stated in Theorems \ref{main result 02} and \ref{main result 03}.

We use standard notation. For $1 \leq p<+\infty,$ $\|u\|_{L^p(\Omega)}=\left(\int_{\Omega}|u|^p d x\right)^{\frac{1}{p}}$ denotes the usual $L^p$-norm. The $N$-dimensional Lebesgue measure of a set $E \subset \mathbb{R}^N$ is denoted by $|E|$. We use "$\rightarrow$" and "$\rightharpoonup$" to denote the strong and weak convergence in the relevant function space, respectively.

\section{Preliminaries}

\subsection{Functional space setup and their embeddings}

In this subsection, following the approach in \cite{DLSV-2025}, we introduce the functional setting relevant to our study.

For any $p \in(1, +\infty)$, we denote the Gagliardo seminorm of a function $u:\mathbb{R}^{N}\to \mathbb{R}$ by
$$
[u]_{s,p}:=\left(c_{N, s, p} \iint_{\mathcal{Q}} \frac{|u(x)-u(y)|^{p}}{|x-y|^{N+p s}} d x d y\right)^{\frac{1}{p}},
$$
where $\mathcal{Q}:=\mathbb{R}^{2 N} \backslash\left(\mathbb{R}^{N} \backslash \Omega\right)^{2}$. 

The space, which plays a central role in our work, is defined as
\begin{equation}\label{whole space}
    \mathcal{W}_{\alpha, \mu, p}(\Omega):= \begin{cases}W^{1,p}(\Omega) & \text { if } \mu \equiv 0,  \\ \mathcal{W}_{\mu,p}(\Omega) & \text { if } \alpha=0, \\ W^{1,p}(\Omega) \cap \mathcal{W}_{\mu,p}(\Omega) & \text { if } \alpha \neq 0 \text { and } \mu \not \equiv 0,\end{cases}
\end{equation}
with the norm
\begin{align}\label{whole norm}
\begin{split}
    &\|u\|_{\alpha, \mu, p}:=\Bigg(\|u\|_{L^{p}(\Omega)}^{p}+\left\||h|^{\frac{1}{p}} u\right\|_{L^{p}(\partial \Omega)}^{p}+\left\||g|^{\frac{1}{p}} u\right\|_{L^{p}\left(\mathbb{R}^{N} \backslash \Omega\right)}^{p}\\
& \quad \quad \quad \quad \quad \quad \quad \quad \quad \quad \quad \quad \quad \quad \quad \quad \quad +\alpha\|\nabla u\|_{L^{p}(\Omega)}^{p}+\frac{1}{2} \int_{(0,1)}[u]_{s,p}^{p} d \mu(s)\Bigg)^{\frac{1}{p}}. 
\end{split}
\end{align}
The space $\mathcal{W}_{\mu, p}(\Omega)$ in \eqref{whole space} is defined as
\begin{equation}\label{space mu p}
\mathcal{W}_{\mu, p}(\Omega):=\left\{u: \mathbb{R}^{N} \rightarrow \mathbb{R} \text { measurable }:\|u\|_{\mu, p}<+\infty\right\}, 
\end{equation}
with the norm
\begin{equation}\label{norm mu p}
\|u\|_{\mu, p}:=\left(\|u\|_{L^{p}(\Omega)}^{p}+\left\||g|^{\frac{1}{p}} u\right\|_{L^{p}\left(\mathbb{R}^{N} \backslash \Omega\right)}^{p}+\frac{1}{2} \int_{(0,1)}[u]_{s,p}^{p} d \mu(s)\right)^{\frac{1}{p}}. 
\end{equation}

Following the ideas in \cite{DLSV-2025}, we prove some properties and embedding results of the spaces introduced above.

\begin{prop}\label{reflex mu p}
The space $\mathcal{W}_{\mu, p}(\Omega)$ endowed with the norm \eqref{norm mu p} is a reflexive Banach space. 
\end{prop}
\begin{proof} The proof proceeds in three steps.

\noindent ($i$) $\|\cdot\|_{\mu, p}$ defines a norm. 

If $\|u\|_{\mu, p}=0$, then we have $\|u\|_{L^{p}(\Omega)}=0$. Consequently, $u=0$ a.e. in $\Omega$. We also have
\begin{equation*}
    \int_{(0,1)}\left(c_{N, s, p} \iint_{\mathcal{Q}} \frac{|u(x)-u(y)|^{p}}{|x-y|^{N+ps}} dx dy\right) d\mu(s)=0,
\end{equation*}
which implies
\begin{equation*}
    \iint_{\mathcal{Q}} \frac{|u(x)-u(y)|^{p}}{|x-y|^{N+p s}} dx d y=0 \quad \text { for any } s \in \operatorname{supp}(\mu).
\end{equation*}
Therefore, we conclude that $|u(x)-u(y)|=0$ for any $(x, y) \in \mathcal{Q}$. In particular, taking $x \in \mathbb{R}^{N} \backslash \Omega$ and $y \in \Omega$, we obtain
\begin{equation*}
    u(x)=u(x)-u(y)=0.
\end{equation*}
Hence, we have $u=0$ for a.e. $x \in \mathbb{R}^{N}$.

\noindent ($ii$) The space $\mathcal{W}_{\mu, p}(\Omega)$ is complete with respect to the norm defined in \eqref{norm mu p}.

Let $\{u_{n}\}$ be a Cauchy sequence in $\mathcal{W}_{\mu, p}(\Omega)$. For the moment, assume that $\{u_{n}\}$ converges to some $u$ a.e. in $\mathbb{R}^{N}$. Then, by applying Fatou’s Lemma and using the fact that $\{u_{n}\}$ is a Cauchy sequence, we conclude that for any $\varepsilon>0$ there exists $N_{\varepsilon}>0$ such that, for any $m \geq N_{\varepsilon}$,
\begin{align*}
 \varepsilon^{p} &\geq \liminf _{n \rightarrow+\infty}\left\|u_{m}-u_{n}\right\|_{\mu,p}^{p} \\
& \geq \liminf _{n \rightarrow+\infty} \int_{\Omega}\left|u_{m}-u_{n}\right|^{p} d x+\liminf _{n \rightarrow+\infty} \int_{\mathbb{R}^{N} \backslash \Omega}|g|\left|u_{m}-u_{n}\right|^{p} d x \\
& +\frac{1}{2} \liminf _{n \rightarrow+\infty} \int_{(0,1)} c_{N, s, p} \int_{\mathcal{Q}} \frac{\left|\left(u_{m}-u_{n}\right)(x)-\left(u_{m}-u_{n}\right)(y)\right|^{p}}{|x-y|^{N+p s}} dx d yd\mu(s)\\
 &\geq  \int_{\Omega}\left|u_{m}-u\right|^{p} dx+\int_{\mathbb{R}^{N} \backslash \Omega}|g|\left|u_{m}-u\right|^{p} dx \\
& +\frac{1}{2} \int_{(0,1)} c_{N, s, p} \iint_{Q} \frac{\left|\left(u_{m}-u\right)(x)-\left(u_{m}-u\right)(y)\right|^{p}}{|x-y|^{N+p s}} dx dy d\mu(s) \\
& = \left\|u_{m}-u\right\|_{\mu,p}^{p}.
\end{align*}
Hence, $\left\{u_m\right\}$ converges to $u$ in $\mathcal{W}_{\mu, p}(\Omega)$, ensuring that the space $\mathcal{W}_{\mu, p}(\Omega)$ is complete. 

It remains to show that $u_{n}\to u$ a.e. in $\mathbb{R}^{N}$. Observe that that $\{u_{n}\}$ is a Cauchy sequence in $L^{p}(\Omega)$, and there exists $u \in L^{p}(\Omega)$ such that, up to a subsequence, $u_{n}\to u$ in $L^{p}(\Omega)$ and $u_{n}\to u$ a.e. in $\Omega$. Hence, there exists a subset $A \subset \mathbb{R}^{N}$ such that
\begin{equation}\label{B.1}
\left|A\right|=0 \quad \text { and } \quad u_{n}(x) \rightarrow u(x) \quad \text { for all } x \in \Omega \backslash A . 
\end{equation}
Given any function $U: \mathbb{R}^{N} \rightarrow \mathbb{R}$, for every $(x, y) \in \mathbb{R}^{2 N}$ and every $s \in(0,1)$, we set
\begin{equation}\label{B.2}
F_{U}(x, y, s):=\frac{(U(x)-U(y)) \chi_{\mathcal{Q}}(x, y)}{|x-y|^{(N+p s) / p}}. 
\end{equation}
Then, we have
\begin{equation*}
    F_{u_{n}}(x, y, s)-F_{u_{m}}(x, y, s)=\frac{\left(u_{n}(x)-u_{n}(y)-u_{m}(x)+u_{m}(y)\right) \chi_{\mathcal{Q}}(x, y)}{|x-y|^{(N+ps) / p}}.
\end{equation*}
Since $\{u_{n}\}$ is a Cauchy sequence, we deduce that for any $\varepsilon>0$ there exists an integer $N_{\varepsilon}$ such that for all indices $n, m \geq N_{\varepsilon}$, we have
\begin{equation*}
    \varepsilon^{p} \geq \int_{(0,1)} c_{N, s, p} \iint_{Q} \frac{\left|\left(u_{n}-u_{m}\right)(x)-\left(u_{n}-u_{m}\right)(y)\right|^{p}}{|x-y|^{N+p s}} dx dy d\mu(s)=\left\|F_{u_{n}}-F_{u_{m}}\right\|_{L^{p}\left(\mathbb{R}^{2 N} \times(0,1)\right)}^{p}.
\end{equation*}
This implies that $F_{u_{n}}$ is a Cauchy sequence in $L^{p}\left(\mathbb{R}^{2 N} \times(0,1), dx dy d\mu\right)$. Therefore, there exists $F\in L^{p}\left(\mathbb{R}^{2 N} \times(0,1), dx dy d\mu\right)$ such that, up to a subsequence, $F_{u_{n}}\to F$ in $L^{p}\left(\mathbb{R}^{2 N} \times(0,1), dx dy d\mu\right)$ and $F_{u_{n}}(x, y, s)\to F(x, y, s)$ a.e. in $\mathbb{R}^{2 N} \times(0,1)$. Hence, there exist subsets $B \subset \mathbb{R}^{2 N}$ and $\Sigma \subset(0,1)$ such that
\begin{align}\label{B.3}
\begin{split}
    & \left|B\right|=0, \quad \mu(\Sigma)=0, \\
& \text { and } \quad F_{u_{n}}(x, y, s) \rightarrow F(x, y, s) \text { for all }(x, y) \in \mathbb{R}^{2 N} \backslash B, \quad s \in(0,1) \backslash \Sigma . 
\end{split}
\end{align}
We now fix $s \in(0,1) \backslash \Sigma$, and for any $x \in \Omega$ define the sets
\begin{align*}
P_{x} & :=\left\{y \in \mathbb{R}^{N}:(x, y) \in \mathbb{R}^{2 N} \backslash B\right\}, \\
M & :=\left\{(x, y) \in \mathbb{R}^{2 N}: x \in \Omega \text { and } y \in \mathbb{R}^{N} \backslash P_{x}\right\}, \\
N & :=\left\{x \in \Omega:\left|\mathbb{R}^{N} \backslash P_{x}\right|=0\right\} .
\end{align*}
We claim that the set $N \setminus A$ is non-empty. First, observe that if $(x, y) \in M$, then $y \in \mathbb{R}^{N} \backslash P_{x}$, which means that $(x, y) \notin \mathbb{R}^{2 N} \backslash B$, and hence $(x, y) \in B$. This means that
\begin{equation}\label{B.4}
M \subseteq B. 
\end{equation}
Combining \eqref{B.3} and \eqref{B.4}, we deduce that $|M|=0$. Therefore, applying Fubini’s theorem yields
\begin{equation*}
    0=|M|=\int_{\Omega}\left|\mathbb{R}^{N} \backslash P_{x}\right| d x.
\end{equation*}
Hence, we have $\left|\mathbb{R}^{N} \backslash P_{x}\right|=0$ for a.e. $x \in \Omega$, and consequently, $|\Omega \backslash N|=0$. Along with \eqref{B.1}, this implies
\begin{equation*}
    \left|\Omega \backslash\left(N \backslash A\right)\right|=\left|(\Omega \backslash N) \cup A\right| \leq|\Omega \backslash N|+\left|A\right|=0 .
\end{equation*}
Thus, in particular, $N \backslash A\neq \emptyset$, as claimed. 

We now fix $x_{0} \in N \backslash A$. Since $x_{0} \in \Omega \backslash A$, using \eqref{B.1}, we get
\begin{equation}\label{B.5.1}
    \lim _{n \rightarrow+\infty} u_{n}\left(x_{0}\right)=u\left(x_{0}\right).
\end{equation}
 In addition, since $x_{0} \in N$, we get $\left|\mathbb{R}^{N} \backslash P_{x_{0}}\right|=0$. This means that for any $y \in P_{x_{0}}$ (i.e., for a.e. $y \in \mathbb{R}^{N}$), we have $\left(x_{0}, y\right) \in \mathbb{R}^{2 N} \backslash B$. Consequently, from \eqref{B.2} and \eqref{B.3}, we deduce that
\begin{equation}\label{B.5}
\lim _{n \rightarrow+\infty} F_{u_{n}}\left(x_{0}, y, s\right)= \lim _{n \rightarrow+\infty}\frac{\left(u_{n}\left(x_{0}\right)-u_{n}(y)\right) \chi_{\mathcal{Q}}\left(x_{0}, y\right)}{\left|x_{0}-y\right|^{(N+ps) / p}}=F\left(x_{0}, y, s\right).
\end{equation}
Moreover, since $\Omega \times\left(\mathbb{R}^{N} \backslash \Omega\right) \subseteq \mathcal{Q}$, it follows from \eqref{B.2} that
\begin{equation}\label{B.5.2}
    F_{u_{n}}\left(x_{0}, y, s\right):=\frac{u_{n}\left(x_{0}\right)-u_{n}(y)}{\left|x_{0}-y\right|^{(N+p s) / p}} \quad \text { for a.e. } y \in \mathbb{R}^{N} \backslash \Omega.
\end{equation}
Hence, combining \eqref{B.5.1}-\eqref{B.5.2}, we obtain
\begin{align}\label{B.L}
\begin{split}
    \lim _{n \rightarrow+\infty} u_{n}(y) & =\lim _{n\rightarrow+\infty}\left(u_{n}\left(x_{0}\right)-\left|x_{0}-y\right|^{(N+ps) / p} F_{u_{n}}\left(x_{0}, y, s\right)\right) \\
& =u\left(x_{0}\right)-\left|x_{0}-y\right|^{(N+ps) / p} F\left(x_{0}, y, s\right)
\end{split}
\end{align}
for a.e. $y \in \mathbb{R}^{N} \backslash \Omega$. Note that the limit in \eqref{B.L} does not depend on $s$. Combining \eqref{B.L} and \eqref{B.1}, up to a change of notation, we conclude that $\{u_{n}\}$ converges a.e. in $\mathbb{R}^{N}$ to some $u$, as desired.

\noindent($iii$) The space $\mathcal{W}_{\mu,p}(\Omega)$ is reflexive. 

To prove reflexivity, we define the space $\mathcal{A}=L^p(\mathcal{Q}\times(0,1), d x d y d\mu(s)) \times L^p(\Omega, d x) \times L^p\left(\mathbb{R}^N \backslash \Omega,|g| d x\right)$ equipped with the norm in \eqref{norm mu p}. This product space is a reflexive Banach space. 

Consider the operator $\mathcal{T}: \mathcal{W}_{\mu,p}(\Omega) \rightarrow \mathcal{A}$ defined as
\begin{equation*}
    \mathcal{T} u=\left[c_{N,s,p}^{\frac{1}{p}}\frac{|u(x)-u(y)|}{|x-y|^{\frac{N}{p}+s}} \chi_{\mathcal{Q}\times (0,1)}(x, y, s), u \chi_{\Omega}, u \chi_{\mathbb{R}^N \backslash \Omega}\right],
\end{equation*}
where $\chi_{\mathcal{M}}(\cdot)$ denotes the characteristic function of a measurable set $\mathcal{M}$. Note that $\mathcal{T}\left(\mathcal{W}_{\mu,p}(\Omega)\right)$ is a closed subspace of $\mathcal{A}$. Moreover, $\mathcal{W}_{\mu,p}(\Omega)$ is a Banach space. Then, using the fact that $\mathcal{T}$ is an isometry from $\mathcal{W}_{\mu,p}(\Omega)$ into $\mathcal{A}$, we conclude that $\mathcal{W}_{\mu,p}(\Omega)$ is a reflexive Banach space as well.
\end{proof}

\begin{prop}\label{Reflex of full}
   The space $\mathcal{W}_{\alpha, \mu, p}(\Omega)$ endowed with the norm \eqref{whole norm} is a reflexive Banach space.
\end{prop}
\begin{proof}
   All the properties of $\mathcal{W}_{\alpha, \mu, p}(\Omega)$ follow from its definition in \eqref{whole space} and Proposition \ref{reflex mu p}. Thus, we only need to show that $\|\cdot\|_{\alpha, \mu, p}$ defines a norm.

If $\mu\equiv0$, then $\|u\|_{\alpha, \mu, p}=0$ implies $\|u\|_{L^{p}(\Omega)}=0$, and hence $u=0$ a.e. in $\Omega$. If $\mu \not \equiv 0$, then we also have
\begin{equation*}
    \int_{(0,1)}\left(c_{N, s, p} \iint_{\mathcal{Q}} \frac{|u(x)-u(y)|^{p}}{|x-y|^{N+p s}} dxdy\right) d\mu(s)=0,
\end{equation*}
which yields
\begin{equation*}
    \iint_{\mathcal{Q}} \frac{|u(x)-u(y)|^{p}}{|x-y|^{N+ps}} dx dy=0 \quad \text { for any } s \in \operatorname{supp}(\mu).
\end{equation*}
Thus, we have $|u(x)-u(y)|=0$ for any $(x, y) \in \mathcal{Q}$. In particular, taking $x \in \mathbb{R}^{n} \backslash \Omega$ and $y \in \Omega$, we obtain
\begin{equation*}
    u(x)=u(x)-u(y)=0,
\end{equation*}
which implies that $u=0$ for a.e. $x \in \mathbb{R}^{N}$. This concludes the proof.
\end{proof}

We now state some results that will be used to prove the embedding theorems for the spaces $\mathcal{W}_{\mu, p}(\Omega)$ and $\mathcal{W}_{\alpha, \mu, p}(\Omega)$.
\begin{lem}\label{tech lem for embed} 
    Let $\Omega$ be an open subset of $\mathbb{R}^{N}$, $\underline{s} \in(0,1)$, and $p\in(1,+\infty)$. Suppose $\underline{s} \leq s_{1} \leq s_{2}<1$. Then, for any measurable function $u: \Omega \rightarrow \mathbb{R}$ we have
\begin{equation*}
    \|u\|_{L^{p}(\Omega)}^{p}+\iint_{\Omega \times \Omega} \frac{|u(x)-u(y)|^{p}}{|x-y|^{N+p s_{1}}} dx dy \leq C\left(\|u\|_{L^{p}(\Omega)}^{p}+\iint_{\Omega \times \Omega} \frac{|u(x)-u(y)|^{p}}{|x-y|^{N+p s_{2}}}dxdy\right),
\end{equation*}
where $C=C(N, \underline{s}, p) \geq 1$.

The exact value of the constant $C$ can be determined as
\begin{equation*}
    C(N, \underline{s}, p):=1+\frac{2^p \omega_{N-1}}{p \underline{s}},
\end{equation*}
where $\omega_{N-1}$ denotes the measure of the unit sphere in $\mathbb{R}^{N}$.
\end{lem}
\begin{proof}
We start by writing
\begin{align*}
\iint_{\Omega \times \Omega}  \frac{|u(x)-u(y)|^{p}}{|x-y|^{N+p s_{1}}} dx dy & =\iint_{(\Omega \times \Omega) \cap\{|x-y|<1\}} \frac{|u(x)-u(y)|^{p}}{|x-y|^{N+p s_{1}}} dx dy\\
&+\iint_{(\Omega \times \Omega) \cap\{|x-y| \geq 1\}} \frac{|u(x)-u(y)|^{p}}{|x-y|^{N+p s_{1}}} dx dy.
\end{align*}
Since $s_{1} \leq s_{2}$, it follows that
\begin{align}\label{tech ineq estim 1}
\begin{split}
    \iint_{(\Omega \times \Omega) \cap\{|x-y|<1\}} \frac{|u(x)-u(y)|^{p}}{|x-y|^{N+p s_{1}}} dx dy & \leq \iint_{(\Omega \times \Omega) \cap\{|x-y|<1\}} \frac{|u(x)-u(y)|^{p}}{|x-y|^{N+p s_{2}}} dx dy  \\
& \leq \iint_{\Omega \times \Omega} \frac{|u(x)-u(y)|^{p}}{|x-y|^{N+p s_{2}}} dx dy.
\end{split}
\end{align}
Moreover, using the facts $N+p s_{1}>N$ and $s_{1} \geq \underline{s}>0$, we deduce
\begin{align}\label{tech ineq estim 2}
\begin{split}
    & \iint_{(\Omega \times \Omega) \cap\{|x-y| \geq 1\}} \frac{|u(x)-u(y)|^{p}}{|x-y|^{N+p s_{1}}} dx dy  \leq \frac{2^p \omega_{N-1}}{p\underline{s}}\|u\|_{L^{p}(\Omega)}^{p}. 
\end{split}
\end{align}
Indeed, let us write
\begin{equation*}
    I=\iint_{(\Omega \times \Omega) \cap\{|x-y| \geq 1\}} \frac{|u(x)-u(y)|^{p}}{|x-y|^{N+p s_{1}}} dx dy.
\end{equation*}
Then we have
\begin{equation}\label{tech ineq tech 1}
    I\leq 2^{p-1}\left(\iint_{(\Omega \times \Omega) \cap\{|x-y| \geq 1\}} \frac{|u(x)|^{p}}{|x-y|^{N+p s_{1}}} dx dy+\iint_{(\Omega \times \Omega) \cap\{|x-y| \geq 1\}} \frac{|u(y)|^{p}}{|x-y|^{N+p s_{1}}} dx dy\right).
\end{equation}
Observe that
\begin{equation}\label{tech ineq tech 2}
    \iint_{(\Omega \times \Omega) \cap\{|x-y| \geq 1\}} \frac{|u(x)|^p}{|x-y|^{N+p s_1}} d x d y=\int_{\Omega}|u(x)|^p\left(\int_{\{y \in \Omega:|x-y| \geq 1\}} \frac{d y}{|x-y|^{N+p s_1}}\right) d x .
\end{equation}
Using the change of variables, we derive
\begin{equation*}
    \int_{\{y \in \Omega:|x-y| \geq 1\}} \frac{d y}{|x-y|^{N+p s_1}}=\int_{(x-\Omega) \cap\{|z| \geq 1\}} \frac{d z}{|z|^{N+p s_1}} \leq \int_{|z| \geq 1} \frac{d z}{|z|^{N+p s_1}}.
\end{equation*}
Moreover, since $N+p s_1>N$ and $s_1 \geq \underline{s}>0$, we get
\begin{equation*}
    \int_{|z| \geq 1} \frac{d z}{|z|^{N+p s_1}}=\omega_{N-1} \int_1^{\infty} \frac{r^{N-1}}{r^{N+p s_1}} d r=\omega_{N-1} \int_1^{\infty} r^{-1-p s_1} d r=\frac{\omega_{N-1}}{p s_1}\leq \frac{\omega_{N-1}}{p \underline{s}},
\end{equation*}
where $\omega_{N-1}$ is the surface measure of the unit sphere in $\mathbb{R}^N$. 

Taking all the above into account, we conclude from \eqref{tech ineq tech 2} that
\begin{equation*}
    \iint_{(\Omega \times \Omega) \cap\{|x-y| \geq 1\}} \frac{|u(x)|^p}{|x-y|^{N+p s_1}} d x d y\leq \frac{\omega_{N-1}}{p \underline{s}}\int_{\Omega}|u(x)|^p d x .
\end{equation*}
Substituting this into \eqref{tech ineq tech 1} yields
\begin{align*}
I\leq 2^{p-1}2\frac{\omega_{N-1}}{p \underline{s}}\int_{\Omega}|u(x)|^p d x= \frac{2^{p} \omega_{N-1}}{p\underline{s}}\|u\|_{L^p(\Omega)}^p.
\end{align*}

Finally, combining \eqref{tech ineq estim 1} and \eqref{tech ineq estim 2}, we get
\begin{align*}
\|u\|_{L^{p}(\Omega)}^{p}&+\iint_{\Omega \times \Omega} \frac{|u(x)-u(y)|^{p}}{|x-y|^{N+p s_{1}}} dx dy \\
& \leq \left(1+\frac{2^p \omega_{N-1}}{p\underline{s}}\right)\|u\|_{L^{p}(\Omega)}^{p}+\iint_{\Omega \times \Omega} \frac{|u(x)-u(y)|^{p}}{|x-y|^{N+p s_{2}}} dx dy \\
& \leq C\left(\|u\|_{L^{p}(\Omega)}^{p}+\iint_{\Omega \times \Omega} \frac{|u(x)-u(y)|^{p}}{|x-y|^{N+p s_{2}}} dx dy\right),
\end{align*}
as desired.
\end{proof}

We now prove an embedding result for the space $\mathcal{W}_{\mu,p}(\Omega)$ defined in \eqref{space mu p}.
\begin{prop}
    Let $\Omega$ be a bounded open set in $\mathbb{R}^{N}$ with Lipschitz boundary. Let $\widetilde{s}_{\sharp}$ be as in \eqref{exp s sharp} and $1<p<N/\widetilde{s}_\sharp$. Let $p_{\widetilde{s}_{\sharp}}^{*}$ be given by
    $$p_{\widetilde{s}_{\sharp}}^*:=\frac{pN}{N-p\widetilde{s}_{\sharp}}.$$
    Then the embedding 
    $$\mathcal{W}_{\mu, p}(\Omega)\hookrightarrow L^q(\Omega)$$ 
    is continuous for all $q \in[1,p_{\widetilde{s}_{\sharp}}^{*}]$ and compact for all $q \in[1,p_{\widetilde{s}_{\sharp}}^{*})$.
\end{prop}

\begin{proof}
First, applying Lemma \ref{tech lem for embed} with $s_{1}:=\widetilde{s}_{\sharp}$ and $s_{2}:=s \in\left[\widetilde{s}_{\sharp}, 1\right)$, we obtain that 
\begin{align}\label{embed for part 111}
\begin{split}
    &\|u\|_{L^{p}(\Omega)}^{p}+\iint_{\Omega \times \Omega} \frac{|u(x)-u(y)|^{p}}{|x-y|^{N+p s}} dx d y\\
    &\quad \quad \quad \quad \quad \quad \quad \quad \quad \quad \geq \frac{1}{C(N,\widetilde{s}_\sharp,p)}\left(\|u\|_{L^{p}(\Omega)}^{p}+\iint_{\Omega \times \Omega} \frac{|u(x)-u(y)|^{p}}{|x-y|^{N+p \widetilde{s}_{\sharp}}} dx dy\right).
\end{split}
\end{align}
From \eqref{exp s sharp}, we may assume, without loss of generality, that there exists $\delta>0$ such that $\mu\left(\left[\widetilde{s}_{\sharp}, 1-\delta\right]\right)>0$. We denote
\begin{equation}\label{embed for part 1111}
    m:=\min _{s \in[\widetilde{s}_{\sharp}, 1-\delta]}\left\{\frac{1}{\mu\left([\widetilde{s}_{\sharp}, 1-\delta]\right)}, \frac{c_{N, s, p}}{2}\right\} .
\end{equation}
Then, using \eqref{embed for part 111} and \eqref{embed for part 1111}, we get
\begin{align}\label{embed for part 2}
\|u\|_{L^{p}(\Omega)}^{p} & +\int_{(0,1)} \frac{c_{N, s, p}}{2} \iint_{\Omega \times \Omega} \frac{|u(x)-u(y)|^{p}}{|x-y|^{N+p s}} dx dy d\mu(s) \nonumber\\
& \geq \|u\|_{L^{p}(\Omega)}^{p}+\int_{[\widetilde{s}_{\sharp}, 1-\delta]} \frac{c_{N, s, p}}{2} \iint_{\Omega \times \Omega} \frac{|u(x)-u(y)|^{p}}{|x-y|^{N+p s}} dx dy d\mu(s) \nonumber\\
& =\int_{[\widetilde{s}_{\sharp}, 1-\delta]}\left(\frac{\|u\|_{L^{p}(\Omega)}^{p}}{\mu\left([\widetilde{s}_{\sharp}, 1-\delta]\right)}+\frac{c_{N, s, p}}{2} \iint_{\Omega \times \Omega} \frac{|u(x)-u(y)|^{p}}{|x-y|^{N+p s}} dx dy\right) d\mu(s)  \\
& \geq  m \int_{[\widetilde{s}_{\sharp}, 1-\delta]}\left(\|u\|_{L^{p}(\Omega)}^{p}+\iint_{\Omega \times \Omega} \frac{|u(x)-u(y)|^{p}}{|x-y|^{N+p s}} dx d y\right) d\mu(s) \nonumber\\
& \geq \frac{m}{C\left(N, \widetilde{s}_{\sharp}, p\right)} \int_{[\widetilde{s}_{\sharp}, 1-\delta]}\left(\|u\|_{L^{p}(\Omega)}^{p}+\iint_{\Omega \times \Omega} \frac{|u(x)-u(y)|^{p}}{|x-y|^{N+p \widetilde{s}_{\sharp}}} dx dy\right) d\mu(s) \nonumber\\
& \geq  C_{1}\|u\|_{W^{\widetilde{s}_\sharp, p}(\Omega)}^{p},\nonumber
\end{align}
for some $C_{1}=C_{1}\left(N, \widetilde{s}_{\sharp}, p, \delta\right)>0$.

On the other hand, by \eqref{norm mu p}, we have that
\begin{equation}\label{embed for part 1}
\|u\|_{L^{p}(\Omega)}^{p}+\int_{(0,1)} \frac{c_{N, s, p}}{2} \iint_{\Omega \times \Omega} \frac{|u(x)-u(y)|^{p}}{|x-y|^{N+p s}} dx dy d\mu(s) \leq \|u\|_{\mu,p}^{p}. 
\end{equation}

Then, combining inequalities \eqref{embed for part 2} and \eqref{embed for part 1} yields
\begin{equation*}
    \|u\|_{W^{\widetilde{s}_\sharp, p}(\Omega)}^{p} \leq C\|u\|_{\mu,p}^{p},
\end{equation*}
for some constant $C>0$.
Hence, we conclude that $\mathcal{W}_{\mu, p}(\Omega)$ is continuously embedded in $W^{\widetilde{s}_{\sharp},p}(\Omega)$, which in turn implies the continuous embedding into $L^q(\Omega)$ for all $q\in[1,p_{\widetilde{s}_{\sharp}}^{*}]$.

The compactness of the embedding is a direct consequence of \cite[Corollary 7.2]{NPV-2012}. This completes the proof.
\end{proof} 

\begin{prop}
    Let $\Omega$ be a bounded open set in $\mathbb{R}^{N}$ with Lipschitz boundary. Assume that $\alpha \neq 0$ and $1<p<N$. Then the embedding 
    $$\mathcal{W}_{\alpha, \mu, p}(\Omega)\hookrightarrow L^q(\Omega)$$ 
    is continuous for all $q \in[1,p^{*}]$ and compact for all $q \in[1,p^{*})$.
\end{prop}
\begin{proof}
  Since $\alpha \neq 0$, by the definition of the norm in \eqref{whole norm}, we have for any $u \in \mathcal{W}_{\alpha, \mu, p}(\Omega)$ and for a suitable constant $C>0$ that
\begin{equation*}
    \|u\|_{W^{1, p}(\Omega)} \leq C\|u\|_{\alpha, \mu, p}.
\end{equation*}

Thus, the embedding into $W^{1,p}(\Omega)$ is continuous. The compact embedding follows from the classical Sobolev embedding.
\end{proof}

For us, it will be useful to use the following embedding, which covers both cases $\alpha=0$ and $\alpha\neq 0$.
\begin{cor}\label{embed full}
Let $\Omega$ be a bounded open set in $\mathbb{R}^{N}$ with Lipschitz boundary. Let $s_{\sharp}$ be as in \eqref{exp s sharp gen} and $1<p<N/s_\sharp$. Then the embedding 
    $$\mathcal{W}_{\alpha, \mu, p}(\Omega)\hookrightarrow L^q(\Omega)$$ 
    is continuous for all $q\in [1, p_{s_{\sharp}}^{*}]$ and compact for all $q\in [1, p_{s_{\sharp}}^{*})$.
\end{cor}

\subsection{Integration by parts formulas}

In this subsection, we state the analogue of the divergence theorem and derive the integration by parts formula.

\begin{lem}Let $u: \mathbb{R}^{N} \rightarrow \mathbb{R}$ be a function satisfying $u \in C^{2}(\Omega)$. Suppose that the following assumptions hold:

\setword{$(i)$}{Word:item11} 
\begin{equation*}
(-\Delta)_{p}^{s} u \in L^{1}(\Omega \times(0,1));
\end{equation*}

\setword{$(ii)$}{Word:item12} The function
\begin{align*}
\begin{split}
  \Omega \times\left(\mathbb{R}^{N} \backslash \Omega\right) \times(0,1) \ni (x, y, s) \mapsto c_{N, s, p} \frac{|u(x)-u(y)|^{p-2}(u(x)-u(y))}{|x-y|^{N+p s}} \\
\end{split}
\end{align*}
belongs to $L^{1}\left(\Omega \times\left(\mathbb{R}^{N} \backslash \Omega\right) \times(0,1)\right)$.

Then,
\begin{equation}\label{div thm tech 3}
\int_{(0,1)} \int_{\Omega}(-\Delta)_{p}^{s} u(x) d x d \mu(s)=-\int_{(0,1)} \int_{\mathbb{R}^{N} \backslash \Omega} \mathscr{N}_{s, p} u(x) d x d\mu(s).
\end{equation}
\end{lem}

\begin{proof}
 Note that under the assumptions \ref{Word:item11} and \ref{Word:item12}, the integrals in \eqref{div thm tech 3} are finite.
Since
\begin{align*}
    \int_{\Omega} \int_{\Omega} \frac{|u(x)-u(y)|^{p-2}(u(x)-u(y))}{|x-y|^{N+p s}} dx dy&=\int_{\Omega} \int_{\Omega} \frac{|u(y)-u(x)|^{p-2}(u(y)-u(x))}{|x-y|^{N+p s}} dy dx\\
    &=-\int_{\Omega} \int_{\Omega} \frac{|u(x)-u(y)|^{p-2}(u(x)-u(y))}{|x-y|^{N+p s}} dxdy,
\end{align*}
we conclude that 
\begin{equation*}
     \int_{\Omega} \int_{\Omega} \frac{|u(x)-u(y)|^{p-2}(u(x)-u(y))}{|x-y|^{N+p s}} dx dy=0,
\end{equation*}
and hence 
\begin{equation*}
    \int_{(0,1)} c_{N, s, p} \int_{\Omega} \int_{\Omega} \frac{|u(x)-u(y)|^{p-2}(u(x)-u(y))}{|x-y|^{N+p s}} dx dy d\mu(s)=0.
\end{equation*}
Now, using \eqref{p frac laplacian}, \eqref{Neumann def}, and \ref{Word:item12} we obtain
\begin{align*}
\int_{(0,1)} \int_{\Omega}(-\Delta)_{p}^{s} u(x) dx d\mu(s) & =\int_{(0,1)} c_{N, s, p} \int_{\Omega} \int_{\mathbb{R}^{N}} \frac{|u(x)-u(y)|^{p-2}(u(x)-u(y))}{|x-y|^{N+p s}} dy dx d\mu(s) \\
& =\int_{(0,1)} c_{N, s, p} \int_{\Omega} \int_{\mathbb{R}^{N} \backslash \Omega} \frac{|u(x)-u(y)|^{p-2}(u(x)-u(y))}{|x-y|^{N+p s}} dy d x d\mu(s) \\
& =\int_{(0,1)} c_{N, s, p} \int_{\mathbb{R}^{N} \backslash \Omega} \int_{\Omega} \frac{|u(x)-u(y)|^{p-2}(u(x)-u(y))}{|x-y|^{N+ps}} dx dy d\mu(s) \\
& =-\int_{(0,1)} \int_{\mathbb{R}^{N} \backslash \Omega} \mathscr{N}_{s, p} u(y) dy d\mu(s),
\end{align*}
concluding the proof.
\end{proof}

\begin{lem}
Let $u, v: \mathbb{R}^{N} \rightarrow \mathbb{R}$ be functions such that $u, v \in C^{2}(\Omega)$. Suppose that the following assumptions hold:

\setword{$(i)$}{Word:item2}
\begin{equation*}
(-\Delta)_{p}^{s} u v \in L^{1}(\Omega \times(0,1)) ;
\end{equation*}

\setword{$(ii)$}{Word:item1}  The function 
\begin{align*}
\begin{split}
    \mathcal{Q} \times(0,1) \ni(x, y, s) \mapsto c_{N, s, p} \frac{|u(x)-u(y)|^{p-2}(u(x)-u(y))(v(x)-v(y))}{|x-y|^{N+ps}} \\
\end{split}
\end{align*}
belongs to $L^{1}(\mathcal{Q} \times(0,1))$;

\setword{$(iii)$}{Word:item3} The function
\begin{align*}
\begin{split}
    \left(\mathbb{R}^{N} \backslash \Omega\right) \times \Omega \times(0,1) \ni(x, y, s) \mapsto c_{N, s, p} \frac{|u(x)-u(y)|^{p-2}(u(x)-u(y)) v(x)}{|x-y|^{N+p s}}
\end{split}
\end{align*}
belongs to $L^{1}\left(\left(\mathbb{R}^{N} \backslash \Omega\right) \times \Omega \times(0,1)\right)$.

Then, the following integration by parts formula holds:
\begin{align}\label{int bp tech 4}
\begin{split}
    \frac{1}{2} \int_{(0,1)} c_{N, s, p} & \iint_{Q} \frac{|u(x)-u(y)|^{p-2}(u(x)-u(y))(v(x)-v(y))}{|x-y|^{N+p s}} dx dy d\mu(s) \\
& =\int_{(0,1)} \int_{\Omega} v(x)(-\Delta)_{p}^{s} u(x) dx d\mu(s)+\int_{(0,1)} \int_{\mathbb{R}^{N} \backslash \Omega} v(x) \mathscr{N}_{s,p} u(x) d x d \mu(s).
\end{split}
\end{align}
\end{lem}

\begin{proof}
Note that assumptions \ref{Word:item2}, \ref{Word:item1}, and \ref{Word:item3} guarantee that the integrals in \eqref{int bp tech 4} are finite.  

Observe that we have
\begin{align*}
\int_{\Omega} \int_{\Omega} & \frac{|u(x)-u(y)|^{p-2}(u(x)-u(y))(v(x)-v(y))}{|x-y|^{N+p s}} dx dy \\
&  \quad \quad \quad \quad \quad \quad \quad \quad =\int_{\Omega} \int_{\Omega} v(x) \frac{|u(x)-u(y)|^{p-2}(u(x)-u(y))}{|x-y|^{N+p s}}dx dy\\
&  \quad \quad \quad \quad \quad \quad \quad \quad -\int_{\Omega} \int_{\Omega} v(y) \frac{|u(x)-u(y)|^{p-2}(u(x)-u(y))}{|x-y|^{N+p s}} dx d y \\
&  \quad \quad \quad \quad \quad \quad \quad \quad =2 \int_{\Omega} \int_{\Omega} v(x) \frac{|u(x)-u(y)|^{p-2}(u(x)-u(y))}{|x-y|^{N+p s}}dx dy.
\end{align*}
Similarly,
\begin{align*}
\int_{\Omega} \int_{\mathbb{R}^{N} \backslash \Omega} & \frac{|u(x)-u(y)|^{p-2}(u(x)-u(y))(v(x)-v(y))}{|x-y|^{N+p s}} dx dy \\
& \quad \quad \quad \quad \quad \quad \quad =\int_{\Omega} \int_{\mathbb{R}^{N} \backslash \Omega} v(x) \frac{|u(x)-u(y)|^{p-2}(u(x)-u(y))}{|x-y|^{N+p s}} dx dy\\
& \quad \quad \quad \quad \quad \quad \quad -\int_{\Omega} \int_{\mathbb{R}^{N} \backslash \Omega} v(y) \frac{|u(x)-u(y)|^{p-2}(u(x)-u(y))}{|x-y|^{N+p s}} dx dy \\
& \quad \quad \quad \quad \quad \quad \quad =\int_{\Omega} \int_{\mathbb{R}^{N} \backslash \Omega} v(x) \frac{|u(x)-u(y)|^{p-2}(u(x)-u(y))}{|x-y|^{N+p s}} dx dy\\
& \quad \quad \quad \quad \quad \quad \quad +\int_{\mathbb{R}^{N} \backslash \Omega} \int_{\Omega} v(x) \frac{|u(x)-u(y)|^{p-2}(u(x)-u(y))}{|x-y|^{N+p s}} dx dy
\end{align*}
and 
\begin{align*}
\int_{\mathbb{R}^{N} \backslash \Omega} \int_{\Omega} & \frac{|u(x)-u(y)|^{p-2}(u(x)-u(y))(v(x)-v(y))}{|x-y|^{N+p s}} dx dy \\
& \quad \quad \quad \quad \quad \quad \quad =\int_{\mathbb{R}^{N} \backslash \Omega} \int_{\Omega} v(x) \frac{|u(x)-u(y)|^{p-2}(u(x)-u(y))}{|x-y|^{N+p s}} dx dy\\
& \quad \quad \quad \quad \quad \quad \quad -\int_{\mathbb{R}^{N} \backslash \Omega} \int_{\Omega} v(y) \frac{|u(x)-u(y)|^{p-2}(u(x)-u(y))}{|x-y|^{N+p s}} dx dy \\
& \quad \quad \quad \quad \quad \quad \quad =\int_{\mathbb{R}^{N} \backslash \Omega} \int_{\Omega} v(x) \frac{|u(x)-u(y)|^{p-2}(u(x)-u(y))}{|x-y|^{N+p s}} dx dy\\
& \quad \quad \quad \quad \quad \quad \quad +\int_{\Omega} \int_{\mathbb{R}^{N} \backslash \Omega} v(x) \frac{|u(x)-u(y)|^{p-2}(u(x)-u(y))}{|x-y|^{N+p s}} dx dy.
\end{align*}
Then, using the above identities and the decomposition $\mathcal{Q}=(\Omega \times \Omega) \cup\left(\Omega \times \mathbb{R}^{N} \backslash \Omega\right) \cup\left(\mathbb{R}^{N} \backslash \Omega \times \Omega\right)$, we obtain
\begin{align*}
& \frac{1}{2} \iint_{Q} \frac{|u(x)-u(y)|^{p-2}(u(x)-u(y))(v(x)-v(y))}{|x-y|^{N+p s}} dx dy \\
& \quad \quad \quad \quad \quad \quad \quad =\int_{\Omega} \int_{\mathbb{R}^{N}} v(x) \frac{|u(x)-u(y)|^{p-2}(u(x)-u(y))}{|x-y|^{N+p s}} dx dy\\
& \quad \quad \quad \quad \quad \quad \quad +\int_{\mathbb{R}^{N} \backslash \Omega} \int_{\Omega} v(x) \frac{|u(x)-u(y)|^{p-2}(u(x)-u(y))}{|x-y|^{N+p s}} dx dy.
\end{align*}
Consequently, using \eqref{p frac laplacian} and \eqref{Neumann def}, we get
\begin{align*}
\int_{(0,1)} \frac{c_{N, s, p}}{2} & \iint_{\mathcal{Q}} \frac{|u(x)-u(y)|^{p-2}(u(x)-u(y))(v(x)-v(y))}{|x-y|^{N+p s}} dx dy d\mu(s) \\
& =\int_{(0,1)} \int_{\Omega} v(x) c_{N, s, p} \int_{\mathbb{R}^{N}} \frac{|u(x)-u(y)|^{p-2}(u(x)-u(y))}{|x-y|^{N+p s}} dy dx d\mu(s) \\
& +\int_{(0,1)} \int_{\mathbb{R}^{N} \backslash \Omega} v(x) c_{N, s, p} \int_{\Omega} \frac{|u(x)-u(y)|^{p-2}(u(x)-u(y))}{|x-y|^{N+p s}} dy dx d\mu(s) \\
& =\int_{(0,1)} \int_{\Omega} v(x)(-\Delta)_p^{s} u(x) d x d \mu(s)+\int_{(0,1)} \int_{\mathbb{R}^{N} \backslash \Omega} v(x) \mathscr{N}_{s, p} u(x) dx d\mu(s),
\end{align*}
as claimed.
\end{proof}

\subsection{Auxiliary results}

We state here some results that play an important role in establishing the existence of solutions to problem \eqref{main problem 1}. Below, we provide a brief review of linking sets and Alexander-Spanier cohomology, referring to \cite{DL-2007, Frigon-1999}.

\begin{definition}\label{link def}
    Let $D, S, A, B$ be four subsets of a metric space $X$ with $S \subseteq D$ and $B \subseteq A$. We say that ($D, S$) links ($A, B$), if $S \cap A=B \cap D=\emptyset$ and, for every deformation $\eta: D \times[0,1] \rightarrow X \backslash B$ satisfying $\eta(S \times[0,1]) \cap A=\emptyset$, it holds that $\eta(D \times\{1\}) \cap A \neq \emptyset$.
\end{definition} 

We will use the following minimax-type theorem to establish the existence of critical points.
\begin{thm}\cite[Theorem 2.2]{DL-2007}\label{link main thm}
    Let $X$ be a complete Finsler manifold of class $C^{1}$ and let $f: X \rightarrow \mathbb{R}$ be a function of class $C^{1}$. Let $D, S, A, B$ be four subsets of $X$, with $S \subseteq D$ and $B \subseteq A$, such that ($D, S$) links ($A, B$) and such that
\begin{equation*}
    \sup _{S} f<\inf _{A} f, \quad \sup _{D} f<\inf _{B} f
\end{equation*}
(we agree that $\sup \emptyset=-\infty$ and $\inf \emptyset=+\infty$). Define
\begin{equation*}
    c=\inf _{\eta \in \mathcal{N}} \sup f(\eta(D \times\{1\})),
\end{equation*}
where $\mathcal{N}$ is the set of deformations $\eta: D \times[0,1] \rightarrow X \backslash B$ with $\eta(S \times[0,1]) \cap A=\emptyset$. Then we have
\begin{equation*}
    \inf _{A} f \leq c \leq \sup _{D} f.
\end{equation*}
Moreover, if $f$ satisfies $(PS)_{c}$, then $c$ is a critical value of $f$.
\end{thm} 

The following is the cohomological counterpart of Definition \ref{link def}.
\begin{definition}\label{link def }
    Consider a field $\mathbb{K}$, a non-negative integer $m$, and subsets $D, S, A, B$ of a metric space $X$ such that $S \subseteq D$ and $B \subseteq A$. If the intersections $S \cap A$ and $B \cap D$ are empty, and the restriction homomorphism $H^{m}(X \backslash B, X \backslash A ; \mathbb{K}) \rightarrow H^{m}(D, S ; \mathbb{K})$ is not identically zero, then we say that ($D, S$) links ($A, B$) cohomologically in dimension $m$ over $\mathbb{K}$. In the special case where $B = \emptyset$, we say that ($D, S$) links $A$ cohomologically in dimension $m$ over $\mathbb{K}$.
\end{definition}

\begin{prop}\label{link index 3}\cite[Proposition 2.4]{DL-2007}
    If ($D, S$) links ($A, B$) cohomologically (in some dimension), then ($D, S$) links ($A, B$).
\end{prop}

\begin{definition}
    Let $X$ be a real normed space. A subset $A$ of $X$ is said to be symmetric if $-u \in A$ whenever $u \in A$. Moreover, $A$ is said to be a cone if $t u \in A$ whenever $u \in A$ and $t>0$.
\end{definition} 

\begin{thm}\label{linking geometry thm}\cite[Theorem 2.8]{DL-2007}
Let $X$ be a real normed space and let $\mathcal{C}_{-},$ $ \mathcal{C}_{+}$ be two cones such that $\mathcal{C}_{-} \cap \mathcal{C}_{+}=\{0\}$, $\mathcal{C}_{+}$ is closed in $X$,  and such that ($X, \mathcal{C}_{-} \backslash\{0\}$) links $\mathcal{C}_{+}$ cohomologically in dimension $m$ over $\mathbb{K}$. Let $r_{-}, r_{+}>0$ and let
\begin{equation*}
    \begin{array}{ll}
D_{-}=\left\{u \in \mathcal{C}_{-}:\|u\| \leq r_{-}\right\}, & S_{-}=\left\{u \in \mathcal{C}_{-}:\|u\|=r_{-}\right\}, \\
D_{+}=\left\{u \in \mathcal{C}_{+}:\|u\| \leq r_{+}\right\}, & S_{+}=\left\{u \in \mathcal{C}_{+}:\|u\|=r_{+}\right\}.
\end{array}
\end{equation*}
Then the following facts hold:

(a) ( $D_{-}, S_{-}$) links $\mathcal{C}_{+}$ cohomologically in dimension $m$ over $\mathbb{K}$;

(b) ( $D_{-}, S_{-}$) links ( $D_{+}, S_{+}$) cohomologically in dimension $m$ over $\mathbb{K}$;

Moreover, let $e \in X$ with $-e \notin \mathcal{C}_{-}$, let
\begin{align*}
& Q=\left\{u+t e: u \in \mathcal{C}_{-}, t \geq 0,\|u+t e\| \leq r_{-}\right\}, \\
& H=\left\{u+t e: u \in \mathcal{C}_{-}, t \geq 0,\|u+t e\|=r_{-}\right\},
\end{align*}
and assume that $r_{-}>r_{+}$. Then the following facts hold:

(c) $\left(Q, D_{-} \cup H\right)$ links $S_{+}$ cohomologically in dimension $m+1$ over $\mathbb{K}$;

(d) $D_{-} \cup H$ links $\left(D_{+}, S_{+}\right)$ cohomologically in dimension $m$ over $\mathbb{K}$.
\end{thm} 

Let $X$ be a real normed space and $A$ a symmetric subset of $X \backslash\{0\}$. We denote by $i(A)$ the $\mathbb{Z}_{2}$-index of $A$, as defined in \cite{FR-1977, FR-1978}. For the definition and properties of this index, we refer to Fadell and Rabinowitz \cite{FR-1978}. 

We will also use the following result.

\begin{cor}\label{link index}\cite[Corollary 2.9]{DL-2007}
 Let $X$ be a real normed space and let $\mathcal{C}_{-},$ $\mathcal{C}_{+}$ be two symmetric cones in $X$ such that $\mathcal{C}_{+}$ is closed in $X,$ $\mathcal{C}_{-} \cap \mathcal{C}_{+}=\{0\}$ and such that
\begin{equation*}
    i\left(\mathcal{C}_{-} \backslash\{0\}\right)=i\left(X \backslash \mathcal{C}_{+}\right)<\infty.
\end{equation*}
Then the assertion (a)-(d) of Theorem \ref{linking geometry thm} hold for $m=i\left(\mathcal{C}_{-} \backslash\{0\}\right)$ and $\mathbb{K}=\mathbb{Z}_{2}$.
\end{cor}

We end this section with the following mountain pass theorem.
\begin{thm}\label{MP thm}\cite[Theorem 6.1]{Struwe} 
Let $X$ be a Banach space with the norm $\|\cdot\|$, and let $F \in C^{1}(X, \mathbb{R})$. Assume that $F$ satisfies the Palais-Smale condition. Suppose that the following assumptions hold:

- $F(0)=0$;

- there exist $\rho,$ $\alpha>0$ such that if $\|u\|=\rho$, then $F(u) \geq \alpha$;

- there exists a function $u_{1} \in X$ such that $\left\|u_{1}\right\| \geq \rho$ and $F\left(u_{1}\right)<\alpha$.

Set
\begin{equation*}
    \Gamma:=\left\{\gamma \in C([0,1] ; X): \gamma(0)=0, \gamma(1)=u_{1}\right\}.
\end{equation*}
Then,
\begin{equation*}
    \beta:=\inf _{\gamma \in \Gamma} \sup _{u \in \gamma} F(u) \geq \alpha
\end{equation*}
is a critical value of $F$.
\end{thm}

\section{The eigenvalue problem: Proof of Theorem \ref{Main result 1}}

In this section, we study the eigenvalue problem associated with the operator $\mathfrak{L}_{\alpha, \mu, p}+I$, where $I(u):=|u|^{p-2}u$. Precisely, for $\lambda\in \mathbb{R}$, we consider the problem
\begin{equation}\label{Eig problem}
\left\{\begin{array}{l}
\mathfrak{L}_{\alpha, \mu, p}(u)+|u|^{p-2}u=\lambda |u|^{p-2}u \quad \text { in } \Omega,  \\
\text { with homogeneous }(\alpha, \mu, p) \text {-Neumann conditions.}
\end{array}\right.
\end{equation}
Since we deal with homogeneous $(\alpha, \mu, p)$-Neumann boundary conditions (see Definition \ref{boundary cond}), we will assume from now on that $h \equiv 0$ and $g \equiv 0$.

Recall that a function $u \in \mathcal{W}_{\alpha, \mu, p}(\Omega)$ is called a weak solution of \eqref{Eig problem} if it satisfies
\begin{align*}
\alpha \int_{\Omega} |\nabla u|^{p-2} \nabla u \nabla v d x&+\int_{(0,1)} \frac{c_{N, s, p}}{2} \iint_{\mathcal{Q}} \frac{\mathcal{A}(u(x)-u(y))(v(x)-v(y))}{|x-y|^{N+p s}} d x d y d \mu(s)\\
&+\int_{\Omega} |u|^{p-2}u v d x=\lambda \int_{\Omega} |u|^{p-2}u v d x
\end{align*}
for all $v \in \mathcal{W}_{\alpha, \mu, p}(\Omega)$. In the above formula, for $1<p<\infty$, we have used the function $\mathcal{A}: \mathbb{R} \rightarrow \mathbb{R}$, defined by
\begin{equation*}
    \mathcal{A}(t)= \begin{cases}|t|^{p-2} t & \text { if } t \neq 0, \\ 0 & \text { if } t=0.\end{cases}
\end{equation*}

If \eqref{Eig problem} admits a weak solution $u \in \mathcal{W}_{\alpha, \mu, p}(\Omega)$, then we say that $\lambda$ is an eigenvalue of \eqref{Eig problem} with homogeneous $(\alpha, \mu, p)$-Neumann boundary conditions, and that $u$ is an associated $\lambda$-eigenfunction. Observe that $\lambda_1=1$ is the first eigenvalue of \eqref{Eig problem}, and all $\lambda_1$-eigenfunctions are constant functions.

Now, for any $u \in \mathcal{W}_{\alpha,\mu,p}(\Omega)$, we define the functional $\mathcal{R}:\mathcal{W}_{\alpha,\mu,p}(\Omega) \rightarrow \mathbb{R}$ by
\begin{equation*}
    \mathcal{R}(u)=\|u\|_{L^{p}(\Omega)}^{p}.
\end{equation*}
In addition, we define the set
\begin{equation*}
    \mathcal{S}=\{u \in \mathcal{W}_{\alpha,\mu,p}(\Omega): \mathcal{R}(u)=1\}.
\end{equation*}
Since $\mathcal{R} \in C^1\left(\mathcal{W}_{\alpha,\mu,p}(\Omega), \mathbb{R}\right)$, it follows that $\mathcal{S}$ is a $C^1$-Finsler manifold.

Next, we define the functional $\mathcal{P}:\mathcal{W}_{\alpha,\mu,p}(\Omega) \rightarrow \mathbb{R}$ by
\begin{align}\label{energy functional J}
\begin{split}
    \mathcal{P}(u) & :=\alpha\|\nabla u\|_{L^p(\Omega)}^p+\frac{1}{2}\int_{(0,1)}[u]_{s, p}^{p}d\mu(s)+\|u\|_{L^{p}(\Omega)}^{p}. 
\end{split}
\end{align}
Note that $\mathcal{P} \in C^1\left(\mathcal{W}_{\alpha,\mu,p}(\Omega), \mathbb{R}\right).$ Denoting by $\widetilde{\mathcal{P}}$ the restriction of $\mathcal{P}$ to $\mathcal{S}$, we conclude that $\lambda>0$ is a critical value of $\widetilde{\mathcal{P}}$ if and only if it is an eigenvalue of \eqref{Eig problem} (see, e.g., \cite[Lemma 4.3]{PAO-2010}).

We will construct a sequence $\{\lambda_{k}\}$ of eigenvalues for problem \eqref{Eig problem}, using the cohomological index theory due to Fadell and Rabinowitz. For this, we denote by $\mathcal{F}$ the family of all nonempty, closed, symmetric subsets of $\mathcal{S}$ and we set for all $k \in \mathbb{N}$
\begin{equation*}
    \mathcal{F}_{k}=\{A \in \mathcal{F}: i(A) \geq k\},
\end{equation*}
where $i(A)$ is the $\mathbb{Z}_{2}$-cohomological index of a set $A$. 

Now we define the sequence $\{\lambda_{k}\}$, setting 
\begin{equation}\label{eigen seq}
\lambda_{k}=\inf _{A \in \mathcal{F}_{k}} \sup _{u \in A} \mathcal{P}(u). 
\end{equation}
Since $\mathcal{F}_{k+1} \subseteq \mathcal{F}_{k}$ for all $k \in \mathbb{N}$, we assert that the sequence $\{\lambda_{k}\}$ is non-decreasing. 

Next, we show that, for every $k \in \mathbb{N}$, the $\lambda_{k}$ defined in \eqref{eigen seq} is an eigenvalue of \eqref{Eig problem}. 

Recall that the functional $\mathcal{P}$ given by \eqref{energy functional J} is said to satisfy the Palais-Smale condition $(PS)_c$ at the level $c\in\mathbb{R}$ if every sequence $\{u_n\} \subset \mathcal{W}_{\alpha, \mu, p}(\Omega)$ such that
\begin{equation*}
\mathcal{P}(u_n)\to c
\quad\text{and}\quad
\mathcal{P}^{\prime}(u_n) \to 0 \,\, \text{in}\,\,\mathcal{W}_{\alpha, \mu, p}(\Omega)^*
\quad\text{as } n\to+\infty,
\end{equation*}
admits a subsequence which converges strongly in $\mathcal{W}_{\alpha, \mu, p}(\Omega)$. 

\begin{prop}\label{PS condition}
    The functional $\widetilde{\mathcal{P}}$ satisfies the Palais-Smale condition at any level $c \in \mathbb{R}$.
\end{prop} 

\begin{proof}
    Consider sequences $\{u_{n}\} \subset \mathcal{S}$ and $\{\mu_{n}\} \subset \mathbb{R}$ such that $\mathcal{P}\left(u_{n}\right) \rightarrow c$ and 
    \begin{equation}\label{convergence ps2}
        \mathcal{P}^{\prime}\left(u_{n}\right)-\mu_{n} \mathcal{R}^{\prime}\left(u_{n}\right) \rightarrow 0 \text{ in } \mathcal{W}_{\alpha, \mu, p}(\Omega)^*
    \end{equation}
    as $n \rightarrow \infty$. Note that we  have
\begin{equation*}
    \left\|u_{n}\right\|_{\alpha, \mu, p}^{p}=\mathcal{P}\left(u_{n}\right)\rightarrow c, \text{ as } n\to \infty.
\end{equation*}
Hence, the sequence $\left\{u_n\right\}$ is bounded in $\mathcal{W}_{\alpha, \mu, p}(\Omega)$. Since $\mathcal{W}_{\alpha, \mu, p}(\Omega)$ is a reflexive Banach space (see Proposition \ref{Reflex of full}), there exists a subsequence, still denoted by $\left\{u_n\right\}$ and a function $u\in \mathcal{W}_{\alpha, \mu, p}(\Omega)$, such that $u_n \rightharpoonup u$ weakly in $\mathcal{W}_{\alpha, \mu, p}(\Omega)$. Then, by Corollary \ref{embed full}, we conclude that $u_n \rightarrow u$, up to a subsequence, in $L^q(\Omega)$ for $1 \leq q<p_{s_{\sharp}}^*$. In particular, we have $u \in \mathcal{S}$. Moreover, testing \eqref{convergence ps2} with $v:=u_n$, we have
\begin{equation*}
    \mu_{n}=\mathcal{P}\left(u_{n}\right)+o(1) \rightarrow c,
\end{equation*}
so the sequence $\{\mu_{n}\}$ is also bounded.

Now, using Hölder’s inequality together with \eqref{convergence ps2}, we obtain
\begin{align}\label{oper L conv 1}
\begin{split}
    |\left\langle \mathcal{P}^{\prime}(u_n), u_{n}-u\right\rangle|  & =\left|\mu_{n}\left\langle \mathcal{R}^{\prime}(u_n), u_{n}-u\right\rangle\right| +o(1)\\
& \leq p\left|\mu_{n}\right|\left\|u_{n}-u\right\|_{L^{p}(\Omega)}\|u_{n}\|_{L^{p}(\Omega)}^{p-1}+o(1) .
\end{split}
\end{align}
Thus, taking limit as $n\to \infty$ in \eqref{oper L conv 1}, we get
\begin{equation*}
    \lim_{n\to \infty} \left\langle \mathcal{P}^{\prime}(u_n), u_n-u\right\rangle =0.
\end{equation*}
Moreover, $u_n \rightharpoonup u$ weakly in $\mathcal{W}_{\alpha, \mu, p}(\Omega)$ implies
\begin{equation*}
    \lim_{n\to \infty}\left\langle\mathcal{P}^{\prime}(u), u_n-u\right\rangle = 0.
\end{equation*}
Consequently, 
\begin{equation}\label{ps main conv}
    \lim_{n\to \infty}\left\langle \mathcal{P}^{\prime}\left(u_n\right)-\mathcal{P}^{\prime}(u), u_n-u\right\rangle =0.
\end{equation}
We set
\begin{equation*}
    \left\langle \mathcal{P}^{\prime}\left(u_n\right)-\mathcal{P}^{\prime}(u), u_n-u\right\rangle:=p(I_1+I_2+I_3),
\end{equation*}
where
\begin{equation*}
    \begin{gathered}
I_1=\alpha \int_{\Omega}\left(\left|\nabla u_n\right|^{p-2} \nabla u_n-|\nabla u|^{p-2} \nabla u\right)\left(\nabla u_n-\nabla u\right) d x, \\
I_2=\int_{\Omega}\left(|u_n|^{p-2}  u_n-|u|^{p-2} u\right)\left( u_n- u\right) d x,
\end{gathered}
\end{equation*}
and
\begin{equation*}
    \begin{gathered}
I_3=\int_{(0,1)}\frac{c_{N,s,p}}{2}\iint_{\mathbb{R}^{2N}} \frac{\mathcal{A}\left(u_n(x)-u_n(y)\right)\left(u_n(x)-u_n(y)-(u(x)-u(y))\right.}{|x-y|^{N+s p}} d x d y d\mu(s)\\
-\int_{(0,1)}\frac{c_{N,s,p}}{2}\iint_{\mathbb{R}^{2N}} \frac{\mathcal{A}(u(x)-u(y))\left(u_n(x)-u_n(y)-(u(x)-u(y))\right.}{|x-y|^{N+s p}} d x d y d\mu(s).
\end{gathered}
\end{equation*}

Recall the following Simon's inequalities, see \cite{Simon-1978}. For all $p \in(1, +\infty)$ there exists $C_p>0$, depending only on $p$, such that
\begin{equation}\label{Simon ineq}
    |\xi-\eta|^p \leq  \begin{cases}C_p \mathcal{F}_p(\xi, \eta), & p \geq 2, \\ C_p \mathcal{F}_p(\xi, \eta)^{p / 2} \cdot\left(|\xi|^p+|\eta|^p\right)^{(2-p) / 2}, & 1<p < 2,\end{cases}
\end{equation}
where
\begin{equation*}
    \mathcal{F}_p(\xi, \eta)=\left(|\xi|^{p-2} \xi-|\eta|^{p-2} \eta\right)(\xi-\eta), \quad \forall \xi, \eta \in \mathbb{R}.
\end{equation*}

For $p \geq 2$, taking $\xi=\nabla u_n,$ $\eta=\nabla u$ in $I_1$ and using inequality \eqref{Simon ineq}, we get
\begin{equation}\label{pg2 i1}
    \alpha\int_{\Omega}\left|\nabla\left(u_n-u\right)\right|^p d x \leq C_p I_1.
\end{equation}
Similarly, taking $\xi= u_n,$ $\eta=u$ in $I_2$, gives 
\begin{equation}\label{pg2 i2}
    \int_{\Omega}\left|u_n-u\right|^p d x \leq C_p I_2.
\end{equation}
Applying again \eqref{Simon ineq} with $\xi=u_n(x)-u_n(y),$ $\eta=u(x)-u(y)$ in $I_3$, we have
\begin{equation}\label{pg2 i3}
   \int_{(0,1)}\frac{c_{N,s,p}}{2} \iint_{\mathbb{R}^{2N}}  \frac{\left|u_n(x)-u_n(y)-(u(x)-u(y))\right|^p}{|x-y|^{N+s p}} d x d yd\mu(s) \leq C_p I_3 .
\end{equation}
Combining \eqref{pg2 i1}, \eqref{pg2 i2} and \eqref{pg2 i3}, we obtain
\begin{equation}\label{pg2 norm}
    \left\|u_n-u\right\|_{\alpha,\mu,p}^p \leq C_p\left(I_1+I_2+I_3\right), \quad \text{ for } p\geq 2.
\end{equation}

For the case $1<p<2$, taking $\xi=\nabla u_n,$ $\eta=\nabla u$ in \eqref{Simon ineq} and substituting in $I_1$, we get
\begin{align}\label{pl2 i1}
    &\alpha \int_{\Omega}\left|\nabla\left(u_n-u\right)\right|^p d x \\
    &\quad \quad \leq \alpha C_p \int_{\Omega}\left(\left|\nabla u_n\right|^{p-2} \nabla u_n-|\nabla u|^{p-2} \nabla u\right)^{\frac{p}{2}}\left(\nabla u_n-\nabla u\right)^{\frac{p}{2}}\left(\left|\nabla u_n\right|^p+|\nabla u|^p\right)^{\frac{2-p}{2}} d x \nonumber.
\end{align}
Applying Hölder's inequality to the right-hand side of \eqref{pl2 i1} with the exponents $\frac{2}{p}$ and $\frac{2}{2-p}$, we obtain
\begin{align}\label{pl2 i11}
\begin{split}
    \alpha \int_{\Omega}\left|\nabla\left(u_n-u\right)\right|^p d x & \leq  C_p\left(\alpha\int_{\Omega} \left(\left|\nabla u_n\right|^{p-2} \nabla u_n-|\nabla u|^{p-2} \nabla u\right)\left(\nabla u_n-\nabla u\right) d x\right)^{\frac{p}{2}} \\
& \times\left(\alpha\int_{\Omega}\left(\left|\nabla u_n\right|^p+|\nabla u|^p\right) d x\right)^{\frac{2-p}{2}} \\
&\leq C_pM^{\frac{2-p}{2}}\left(\alpha\int_{\Omega} \left(\left|\nabla u_n\right|^{p-2}\nabla u_n-|\nabla u|^{p-2}\nabla u\right)\left(\nabla u_n-\nabla u\right) d x\right)^{\frac{p}{2}} \\
&= C I_1^{\frac{p}{2}}, 
\end{split}
\end{align}
where $C=C_pM^{\frac{2-p}{2}}$ and $\alpha\left(\left\|\nabla u_n\right\|_p^p+\left\|\nabla u\right\|_p^p \right)\leq M$ for some nonnegative constant $M$.

Proceeding as above, we also obtain
\begin{equation}\label{pl2 i2}
   \int_{\Omega}\left|\left(u_n-u\right)\right|^p d x \leq C I_2^{\frac{p}{2}}
\end{equation}
and
\begin{equation}\label{pl2 i3}
    \int_{(0,1)}\frac{c_{N,s,p}}{2}\iint_{\mathbb{R}^{2N}}\frac{\left|\left(u_n-u\right)(x)-\left(u_n-u\right)(y)\right|^p}{|x-y|^{N+s p}} d x d yd\mu(s) \leq C I_3^{\frac{p}{2}}.
\end{equation}
Now, combining \eqref{pl2 i11}, \eqref{pl2 i2} and \eqref{pl2 i3}, using the convexity of the map $t \mapsto t^{\frac{2}{p}}$, we get
\begin{equation}\label{pl2 norm}
    \left\|u_n-u\right\|_{\alpha,\mu, p}^2 \leq C\left(I_1+I_2+I_3\right), \quad \text{ for } 1<p<2.
\end{equation}
  
Finally, using \eqref{ps main conv}, \eqref{pg2 norm} and \eqref{pl2 norm}, we arrive at
\begin{equation*}
    \lim_{n\to \infty}\left\|u_n-u\right\|_{\alpha, \mu, p}^p = 0.
\end{equation*}
This concludes the proof.
\end{proof} 

\begin{proof}[Proof of Theorem \ref{Main result 1}]
In order to prove that, for any $k \in \mathbb{N}, \lambda_k$ defined in \eqref{eigen seq} is an eigenvalue of \eqref{Eig problem}, it suffices to show that $\lambda_k$ is a critical value of $\widetilde{\mathcal{P}}$. Arguing by contradiction, we assume that $\lambda_{k}$ is a regular value of $\widetilde{\mathcal{P}}$. By Proposition \ref{PS condition}, we know that $\widetilde{\mathcal{P}}$ satisfies the Palais-Smale condition.  Therefore, applying \cite[Theorem 2.5]{Bonnet-1993}, we conclude that there exist a real $\varepsilon>0$ and an odd homeomorphism $\eta: \mathcal{S} \rightarrow \mathcal{S}$ such that 
\begin{equation}\label{tech for lambda}
    \mathcal{P}(\eta(u)) \leq \lambda_{k}-\varepsilon
\end{equation}
for all $u \in \mathcal{S}$ with $\mathcal{P}(u) \leq \lambda_{k}+\varepsilon$.

By the definition of $\lambda_{k}$ in \eqref{eigen seq}, we can choose a set $A \in \mathcal{F}_{k}$ such that $\sup _{A} \mathcal{P}<\lambda_{k}+\varepsilon$. We set $B=\eta(A)$. Since $\eta$ is an odd homeomorphism, the set $B$ is closed, symmetric, and nonempty, and hence, $B \in \mathcal{F}$. Moreover, since $i(B)=i(\eta(A))\geq i(A) \geq k$, we have $B \in \mathcal{F}_k$. On the other hand, by \eqref{tech for lambda}, we have $\sup _{B} \mathcal{P} \leq \lambda_{k}-\varepsilon$, which contradicts \eqref{eigen seq}. Therefore, every $\lambda_{k}$ is an eigenvalue of \eqref{Eig problem}.

Finally, since $i(\mathcal{S})=\infty$ and $\sup_\mathcal{S} \mathcal{P}=\infty$, we immediately conclude that $\lambda_k \rightarrow \infty$.
\end{proof}

\begin{remark}
    All the results of this section remain valid if we consider the eigenvalue problem:
\begin{equation*}
\left\{\begin{array}{l}
\mathfrak{L}_{\alpha, \mu, p}(u) =\lambda |u|^{p-2}u \quad \text { in } \Omega,  \\
\text { with homogeneous }(\alpha, \mu, p) \text {-Neumann conditions.}
\end{array}\right.
\end{equation*}
In this case, taking
  \begin{equation}\label{ad after com}
      \mathcal{P}(u) :=\alpha\|\nabla u\|_{L^p(\Omega)}^p+\frac{1}{2}\int_{(0,1)}[u]_{s, p}^{p}d\mu(s)
  \end{equation} 
in \eqref{energy functional J}, one obtains a diverging sequence of eigenvalues with
    \begin{equation*}
        0=\widetilde{\lambda}_1\leq \widetilde{\lambda}_2 \leq ... \to \infty.
    \end{equation*}
The proof follows the same arguments as in the previous case. The only difference is that the functional $\mathcal{P}$ defined in \eqref{ad after com} does not control the full norm of $\mathcal{W}_{\alpha,\mu,p}(\Omega)$ due to the absence of the $L^p$-term. To overcome this difficulty, we consider the operator $\mathcal{L}: \mathcal{W}_{\alpha, \mu, p}(\Omega) \rightarrow \mathcal{W}_{\alpha, \mu, p}(\Omega)^*$ defined by
\begin{align}\label{ad after com 1}
\begin{split}
     \langle \mathcal{L}(u), v\rangle & :=\int_{\Omega}|u|^{p-2} u v d x+\alpha \int_{\Omega}|\nabla u|^{p-2} \nabla u \nabla v d x  \\
& \quad +\int_{(0,1)}\frac{c_{N,s,p}}{2}\iint_{\mathcal{Q}} \frac{\mathcal{A}(u(x)-u(y))(v(x)-v(y))}{|x-y|^{N+p s}} d x d yd\mu(s)
\end{split}
\end{align}
for all $u, v \in \mathcal{W}_{\alpha, \mu, p}(\Omega)$. Then, using \eqref{convergence ps2}, \eqref{ad after com 1}, and the weak convergence in $\mathcal{W}_{\alpha,\mu,p}(\Omega)$, we obtain 
\begin{equation*}
    \lim_{n\to \infty}\left\langle \mathcal{L}\left(u_n\right)-\mathcal{L}(u), u_n-u\right\rangle =0
\end{equation*}
instead of \eqref{ps main conv}. The rest of the proof proceeds exactly as before.
\end{remark}

\section{Existence results: Proofs of Theorems \ref{main result 02} and \ref{main result 03}}

In this section, we prove existence results for problem \eqref{main problem 1}, depending on the position of the parameter $\lambda$ with respect to the first eigenvalue of the eigenvalue problem \eqref{Eig problem}. 

The definition of a weak solution to problem \eqref{main problem 1} is stated as follows: 
\begin{definition}
We say that $u \in \mathcal{W}_{\alpha, \mu, p}(\Omega)$ is a weak solution of problem \eqref{main problem 1} if, for any $v \in \mathcal{W}_{\alpha, \mu, p}(\Omega)$,
\begin{align*}
\alpha \int_{\Omega} |\nabla u|^{p-2} \nabla u \nabla v d x&+\int_{(0,1)} \frac{c_{N, s, p}}{2} \iint_{\mathcal{Q}} \frac{\mathcal{A}(u(x)-u(y))(v(x)-v(y))}{|x-y|^{N+p s}} d x d y d \mu(s)\\
&+\int_{\Omega} |u|^{p-2}u v d x=\lambda \int_{\Omega} |u|^{p-2}u v d x+\int_{\Omega} f(x, u) v d x. 
\end{align*}
\end{definition}

Let us consider the sequence of eigenvalues $\{\lambda_{k}\}$ of problem \eqref{Eig problem} defined in \eqref{eigen seq}. For each $\lambda_{k}$, we define the cones
\begin{align}\label{cone -}
\begin{split}
    \mathcal{C}_{k}^{-}:=\Bigg\{u \in \mathcal{W}_{\alpha, \mu, p}(\Omega): \|u\|_{\alpha,\mu,p}^p \leq \lambda_{k} \int_{\Omega}|u|^{p} d x\Bigg\} 
\end{split}
\end{align}
and
\begin{align}\label{cone +}
\begin{split}
    \mathcal{C}_{k}^{+}:=\Bigg\{u \in \mathcal{W}_{\alpha, \mu, p}(\Omega): \|u\|_{\alpha,\mu,p}^p \geq \lambda_{k+1} \int_{\Omega}|u|^{p} d x\Bigg\} .
\end{split}
\end{align}

The following result is the analogue of \cite[Theorem 3.2]{DL-2007} in our setting, and its proof follows the same steps. See, for example, the proof of \cite[Lemma 2.15]{LY-2013}.
\begin{thm}\label{link index 2}
    Let $k \geq 1$ be such that $\lambda_{k}<\lambda_{k+1}$, then we have
\begin{equation*}
    i\left(\mathcal{C}_{k}^{-} \backslash\{0\}\right)=i\left(\mathcal{W}_{\alpha,\mu,p} \backslash \mathcal{C}_{k}^{+}\right)=k.
\end{equation*}
\end{thm}

\subsection{Existence of a linking solution}

Throughout this subsection, we assume that $\lambda\geq\lambda_1=1$. Since the sequence of eigenvalues is divergent, we can assume that there exists $k \geq 1$ such that 
\begin{equation}\label{Link cond on lamb}
    \lambda_{k} \leq  \lambda<\lambda_{k+1}.
\end{equation}

We define the functional $\mathcal{I}: \mathcal{W}_{\alpha, \mu, p}(\Omega) \rightarrow \mathbb{R}$, associated to problem \eqref{main problem 1}, as 
\begin{align}\label{functional MP}
\begin{split}
    \mathcal{I}(u)&:=\frac{\alpha}{p} \int_{\Omega}|\nabla u|^{p} d x+\int_{(0,1)} \frac{c_{N, s, p}}{2p} \iint_{\mathcal{Q}} \frac{|u(x)-u(y)|^{p}}{|x-y|^{N+p s}} d x d y d \mu(s)+\frac{1}{p} \int_{\Omega}|u|^{p} d x\\
&-\frac{\lambda}{p} \int_{\Omega}|u|^{p} d x-\int_{\Omega} F(x, u) d x .
\end{split}
\end{align}
Note that the functional $\mathcal{I}$ is of class $C^{1}$ and for all $u, v \in \mathcal{W}_{\alpha, \mu, p}(\Omega)$ we have 
\begin{align}\label{functional MP deriv}
\begin{split}
    \left\langle \mathcal{I}^{\prime}(u), v\right\rangle&=  \alpha \int_{\Omega} |\nabla u|^{p-2}\nabla u \nabla v d x\\
    &+\int_{(0,1)} \frac{c_{N, s, p}}{2} \iint_{\mathcal{Q}} \frac{\mathcal{A}(u(x)-u(y))(v(x)-v(y))}{|x-y|^{N+p s}} d x d y d \mu(s) \\
&+ \int_{\Omega} |u|^{p-2}u v d x-\lambda \int_{\Omega} |u|^{p-2}u v d x-\int_{\Omega} f(x, u) v d x .
\end{split}
\end{align}
Moreover, the critical points of $\mathcal{I}$ are weak solutions of problem \eqref{main problem 1}.

\begin{prop}\label{PS for linking}
   Let $\lambda$ be as in \eqref{Link cond on lamb}, and suppose that $f$ fulfills the conditions \ref{Word:f1}-\ref{Word:f4}. Then the functional $\mathcal{I},$ defined in \eqref{functional MP}, satisfies the $(P S)_{c}$ condition for every $c \in \mathbb{R}$.
\end{prop} 

\begin{proof}  
Let $c \in \mathbb{R}$ and let $\{u_{n}\}\subset \mathcal{W}_{\alpha, \mu, p}(\Omega)$ be a sequence such that
\begin{equation}\label{ps mp c1}
\lim _{n \rightarrow+\infty} \mathcal{I}\left(u_{n}\right)=c 
\end{equation}
and
\begin{equation}\label{ps mp c2}
\lim _{n \rightarrow+\infty} \sup _{\substack{v \in \mathcal{W}_{\alpha, \mu, p}(\Omega) \\\|v\|_{\alpha, \mu, p}=1}}\left|\left\langle \mathcal{I}^{\prime}\left(u_{n}\right), v\right\rangle\right|=0. 
\end{equation}
Observe that we have
\begin{equation}\label{functional MP equiv}
\mathcal{I}(u_n) =\frac{1}{ p}\|u_n\|_{\alpha, \mu, p}^{p}-\frac{\lambda}{p} \int_{\Omega}|u_n|^{p} d x-\int_{\Omega} F(x, u_n) d x.
\end{equation}
We fix $k \in(p, \gamma)$, where $\gamma$ is the constant from \ref{Word:f3}. By \eqref{ps mp c1} and \eqref{ps mp c2}, we have
\begin{equation}\label{ps fpr mp tech 1}
k \mathcal{I}\left(u_{n}\right)-\left\langle \mathcal{I}^{\prime}\left(u_{n}\right), u_{n}\right\rangle \leq C_1+C_2\|u_{n}\|_{\alpha,\mu,p} 
\end{equation}
for some $C_1, C_2>0$ and all $n \in \mathbb{N}$. On the other hand, by \eqref{functional MP deriv} and \eqref{functional MP equiv}, we have
\begin{align}\label{ps for mp tech 7}
&k \mathcal{I}\left(u_{n}\right)-\left\langle \mathcal{I}^{\prime}\left(u_{n}\right), u_{n}\right\rangle 
 =k\left(\frac{1}{p}\|u_n\|_{\alpha, \mu, p}^{p}-\frac{\lambda}{p}\int_{\Omega}|u_n|^pdx-\int_{\Omega} F(x,u_n)dx\right) \nonumber \\
&-\left(\|u_n\|_{\alpha, \mu, p}^{p}-\lambda \int_{\Omega}|u_n|^pdx-\int_{\Omega} f(x,u_n)u_ndx\right)  \nonumber \\
&=\left(\frac{k}{p}-1 \right)\|u_n\|_{\alpha, \mu, p}^{p}-\left(\lambda\left(\frac{k}{p}-1\right)\right)\int_{\Omega}|u_n|^pdx+\int_{\Omega}(f(x, u_n)u_n -kF(x, u_n))dx  \nonumber \\
&=\left(\frac{k}{p}-1 \right)\|u_n\|_{\alpha, \mu, p}^{p}-\left(\lambda\left(\frac{k}{p}-1\right)\right)\int_{\Omega}|u_n|^pdx\\
&+\int_{\Omega}\left(f\left(x, u_n\right) u_n-\gamma F\left(x, u_n\right)\right) d x+(\gamma-k) \int_{\Omega} F\left(x, u_n\right) d x. \nonumber 
\end{align}
We consider the term
\begin{align}\label{ps for mp tech 6}
\begin{split}
    \int_{\Omega}\left(f\left(x, u_n\right) u_n-\gamma F\left(x, u_n\right)\right) d x&=\int_{\Omega \cap{\left\{\left|u_n\right| > R\right\}}}\left(f\left(x, u_n\right) u_n-\gamma F\left(x, u_n\right)\right) d x\\
    &+\int_{\Omega \cap{\left\{\left|u_n\right| \leq R\right\}}}\left(f\left(x, u_n\right) u_n-\gamma F\left(x, u_n\right)\right) d x,
\end{split}
\end{align}
where $R\geq0$ is the constant from \ref{Word:f3}.
Applying \ref{Word:f3} for $\left|u_n\right|>R$ we have that
\begin{equation}\label{ps for mp tech 4}
    \int_{\Omega \cap{\left\{\left|u_n\right| > R\right\}}}\left(f\left(x, u_n\right) u_n-\gamma F\left(x, u_n\right)\right) d x \geq 0.
\end{equation}
On the other hand, using \ref{Word:f1} and \eqref{estim for f ex}, for $|t| \leq R$, we have
\begin{equation}\label{ps fpr mp tech 3}
    |f(x, t) t-\gamma F(x, t)| \leq C_R
\end{equation}
for some constant $C_R\geq 0$. Using \eqref{ps fpr mp tech 3}, we obtain
\begin{equation}\label{ps for mp tech 5}
    \left|\int_{\Omega \cap{\left\{\left|u_n\right| \leq R\right\}}}\left(f\left(x, u_n\right) u_n-\gamma F\left(x, u_n\right)\right) d x \right| \leq C_R|\Omega|=\widetilde{C}_R
\end{equation}
for some $\widetilde{C}_R\geq0$. Combining \eqref{ps for mp tech 4} and \eqref{ps for mp tech 5} in \eqref{ps for mp tech 6}, and substituting the result into \eqref{ps for mp tech 7}, we derive 
\begin{align}\label{ps for mp tech 8}
 k \mathcal{I}\left(u_{n}\right)-\left\langle \mathcal{I}^{\prime}\left(u_{n}\right), u_{n}\right\rangle &\geq \left(\frac{k}{p}-1 \right)\|u_n\|_{\alpha, \mu, p}^{p}-\left(\lambda\left(\frac{k}{p}-1\right)\right)\int_{\Omega}|u_n|^pdx\\
&+(\gamma-k)\int_{\Omega}F(x, u_n)dx-\widetilde{C}_R.\nonumber
\end{align}
Applying Hölder and Young inequalities, we have, for any $\varepsilon>0$, that
\begin{equation}\label{ps for mp tech 88}
    \|u_n\|_{p}^{p} \leq |\Omega|^{1-\frac{p}{\tilde{\gamma}}}\|u\|_{\tilde{\gamma}}^p\leq \varepsilon\|u_n\|_{\widetilde{\gamma}}^{\widetilde{\gamma}}+C_{\varepsilon}|\Omega| .
\end{equation}
Further, using \ref{Word:f4} and \eqref{ps for mp tech 88}, we get from \eqref{ps for mp tech 8} that
\begin{align*}
& k \mathcal{I}\left(u_{n}\right)-\left\langle \mathcal{I}^{\prime}\left(u_{n}\right), u_{n}\right\rangle \\
& \geq \left(\frac{k}{p}-1 \right)\|u_n\|_{\alpha, \mu, p}^{p}-\left(\lambda\left(\frac{k}{p}-1\right)\right)\int_{\Omega}|u_n|^pdx +(\gamma-k) a_{3} \int_{\Omega}\left|u_{n}\right|^{\widetilde{\gamma}} d x\\
&-(\gamma-k)  \int_{\Omega}a_{4}(x) d x-\widetilde{C}_R\\
& \geq\left(\frac{k}{p}-1 \right)\|u_n\|_{\alpha, \mu, p}^{p}+\left[(\gamma-k) a_{3}-\lambda\varepsilon\left(\frac{k}{p}-1\right)\right] \int_{\Omega}\left|u_{n}\right|^{\widetilde{\gamma}} d x-\widetilde{C}_{\varepsilon}
\end{align*}
for some $\widetilde{C}_{\varepsilon}>0$. Taking $\varepsilon$ small enough, this yields
\begin{equation}\label{ps fpr mp tech 2}
    k \mathcal{I}\left(u_{n}\right)-\left\langle \mathcal{I}^{\prime}\left(u_{n}\right), u_{n}\right\rangle \geq\left(\frac{k}{p}-1\right)\left\|u_{n}\right\|_{\alpha, \mu, p}^{p}-\widetilde{C}_{\varepsilon} .
\end{equation}
Thus, from \eqref{ps fpr mp tech 1} and \eqref{ps fpr mp tech 2}, we deduce that the sequence $\{u_{n}\}$ is bounded in $\mathcal{W}_{\alpha, \mu, p}(\Omega)$. Hence, we may assume that there exists $u\in \mathcal{W}_{\alpha, \mu, p}(\Omega)$ such that 
\begin{equation}\label{weak conv 1}
    u_{n} \rightharpoonup u \text{ in } \mathcal{W}_{\alpha, \mu, p}(\Omega)
\end{equation}
and 
\begin{equation}\label{strong lp conv 1}
    u_{n} \rightarrow u \text{ in } L^{p}(\Omega),
\end{equation}
up to a subsequence, as $n \rightarrow \infty$. 

By \eqref{ps mp c2} and \eqref{weak conv 1}, we have that
\begin{equation*}
    \left\langle \mathcal{I}^{\prime}\left(u_{n}\right), u_{n}-u\right\rangle \rightarrow 0.
\end{equation*}
Moreover, we have
\begin{align*}
\left\langle \mathcal{P}^{\prime}\left(u_{n}\right), u_{n}-u\right\rangle & =\left\langle \mathcal{I}^{\prime}\left(u_{n}\right), u_{n}-u\right\rangle+\lambda \int_{\Omega}\left|u_{n}\right|^{p-2} u_{n}\left(u_{n}-u\right) d x\\
&+\int_{\Omega} f\left(x, u_{n}\right)\left(u_{n}-u\right) d x,
\end{align*}
where $\mathcal{P}$ is introduced in \eqref{energy functional J}.
Using \ref{Word:f1} and \eqref{strong lp conv 1}, we obtain
$$
\int_{\Omega}\left|u_{n}\right|^{p-2} u_{n}\left(u_{n}-u\right) d x \rightarrow 0
$$
and
$$
\int_{\Omega} f\left(x, u_{n}\right)\left(u_{n}-u\right) d x \rightarrow 0.
$$
Consequently, we have 
\begin{equation*}
    \left\langle \mathcal{P}^{\prime} \left(u_{n}\right), u_{n}-u\right\rangle \rightarrow 0 \quad \text{ as } n \rightarrow \infty.
\end{equation*}
Moreover, $u_n \rightharpoonup u$ weakly in $\mathcal{W}_{\alpha, \mu, p}(\Omega)$ implies
\begin{equation*}
    \left\langle \mathcal{P}^{\prime}\left(u\right), u_{n}-u\right\rangle \rightarrow 0 \quad \text{ as } n \rightarrow \infty.
\end{equation*}
Then, by using the same argument as in the proof of Proposition \ref{PS condition}, we conclude that $u_{n} \rightarrow u$ in $\mathcal{W}_{\alpha, \mu, p}(\Omega)$, as desired.
\end{proof}

\begin{proof}[Proof of Theorem \ref{main result 02}]
    We begin by verifying the following geometric properties of the functional $\mathcal{I}$. 
    
    Let $\mathcal{C}_{k}^{-}$ and $\mathcal{C}_{k}^{+}$ be as in \eqref{cone -} and \eqref{cone +}, respectively. 
Taking any $u \in \mathcal{C}_{k}^{+}$, and using the inequalities in \eqref{cone +}, \eqref{estim for f ex}, and Corollary \ref{embed full}, we have that
\begin{align}\label{linking geom tech 11}
\mathcal{I}(u) &=\frac{\alpha}{p} \int_{\Omega}|\nabla u|^{p} d x+\int_{(0,1)} \frac{c_{N, s, p}}{2p} \iint_{\mathcal{Q}} \frac{|u(x)-u(y)|^{p}}{|x-y|^{N+p s}} d x d y d \mu(s) +\frac{1}{p} \int_{\Omega}|u|^{p} d x\nonumber\\
&-\frac{\lambda}{p} \int_{\Omega}|u|^{p} d x-\int_{\Omega} F(x, u) d x\nonumber\\
&=\frac{1}{p}\left(\alpha \int_{\Omega}|\nabla u|^{p} d x+\int_{(0,1)} \frac{c_{N, s, p}}{2} \iint_{\mathcal{Q}} \frac{|u(x)-u(y)|^{p}}{|x-y|^{N+p s}} d x d y d \mu(s)+\int_{\Omega}|u|^{p} d x\right)\nonumber\\
&-\frac{\lambda}{p} \int_{\Omega}|u|^{p} d x-\int_{\Omega} F(x, u) d x\\
& \geq \frac{1}{p}\|u\|_{\alpha,\mu,p}^{p}-\frac{ \lambda}{ p} \int_{\Omega}|u|^{p} d x-\frac{\varepsilon}{p}\int_{\Omega}|u|^{p} d x-\delta(\varepsilon) \int_{\Omega}|u|^{q} d x\nonumber \\
& \geq \frac{1}{p}\|u\|_{\alpha,\mu,p}^{p}-\frac{1}{ p \lambda_{k+1}}(\lambda+\varepsilon)\|u\|_{\alpha,\mu,p}^{p}-\delta(\varepsilon) \int_{\Omega}|u|^{q} d x \nonumber\\
& \geq \frac{1}{p}\left(1-\frac{ \lambda+\varepsilon}{\lambda_{k+1}}\right)\|u\|_{\alpha,\mu,p}^{p}-c\delta(\varepsilon)\|u\|_{\alpha,\mu,p}^{q}\nonumber\\
&=\|u\|^p_{\alpha,\mu,p}\left(\frac{1}{p}-\frac{ \lambda+\varepsilon}{p\lambda_{k+1}}-c\delta(\varepsilon)\|u\|_{\alpha,\mu,p}^{q-p}\right)\nonumber
\end{align}
for $c=c\left(N, \Omega, s_{\sharp},p\right)>0$.

Taking
\begin{equation*}
    \varepsilon \in\left(0, \lambda_{k+1}-\lambda\right) \quad \text { and } \quad r_{+} \in\left(0,\left(\frac{1}{pc\delta(\varepsilon)}-\frac{\lambda+\varepsilon}{p \lambda_{k+1}c\delta(\varepsilon)}\right)^{1 /(q-p)}\right)
\end{equation*}
we obtain that
$$
\frac{1}{p}-\frac{\lambda+\varepsilon}{p \lambda_{k+1}}-c \delta(\varepsilon) r_+^{q-p}>0.
$$
Thus, from \eqref{linking geom tech 11} we get 
$$
\inf _{\substack{u \in C_k^+ \\\|u\|_{\alpha, \mu, p}=r_+}} \mathcal{I}(u) \geq r_+^{p}\left(\frac{1}{p}-\frac{\lambda+\varepsilon}{p \lambda_{k+1}}-c \delta(\varepsilon) r_+^{q-p}\right)=: \theta>0.
$$
Hence, we have that there exist $r_{+}>0$ and $\theta>0$ such that, if $\|u\|_{\alpha, \mu, p}=r_{+}$, then $\mathcal{I}(u) \geq \theta$.

Using \ref{Word:f3}, \ref{Word:f5}, and \eqref{Link cond on lamb}, for all $u \in \mathcal{C}_{k}^{-}$, we derive
\begin{align*}
\mathcal{I}(u) & =\frac{1}{p}\|u\|_{\alpha, \mu, p}^p-\frac{\lambda}{p} \int_{\Omega}|u|^p d x-\int_{\Omega} F(x, u) d x \leq\frac{\lambda_k-\lambda}{p}\int_{\Omega}|u|^p d x\leq 0.
\end{align*}
Moreover, taking $e \in \mathcal{W}_{\alpha,\mu,p} \backslash \mathcal{C}_{k}^{-}$ and using \ref{Word:f4}, we obtain, for each $u \in \mathcal{C}_{k}^{-}$ and $t>0$, that 
\begin{align*}
\mathcal{I}(u+t e) &=\frac{\alpha}{p} \int_{\Omega}|\nabla (u+te)|^{p} d x+\int_{(0,1)} \frac{c_{N, s, p}}{2p} \iint_{\mathcal{Q}} \frac{|(u+te)(x)-(u+te)(y)|^{p}}{|x-y|^{N+p s}} d x d y d \mu(s)\\
&+\frac{1}{p} \int_{\Omega}|u+te|^{p} d x-\frac{\lambda}{p} \int_{\Omega}|u+te|^{p} d x-\int_{\Omega} F(x, u+te) d x\\
&\leq \frac{2^{p-1}\alpha}{p} \left(\int_{\Omega}|\nabla u|^p d x+t^p \int_{\Omega}|\nabla e|^p d x\right)\\
&+\frac{2^{p-1}}{2p}\Bigg(\int_{(0,1)} c_{N, s, p} \iint_{\mathcal{Q}} \frac{|u(x)-u(y)|^{p}}{|x-y|^{N+p s}} d x d y d \mu(s)\\
&+t^p\int_{(0,1)} c_{N, s, p} \iint_{\mathcal{Q}} \frac{|e(x)-e(y)|^{p}}{|x-y|^{N+p s}} d x d y d \mu(s)\Bigg)\\
&+\frac{(1-\lambda) t^p}{p} \int_{\Omega}\left|\frac{u}{t}+e\right|^{p} d x-a_{3} t^{\widetilde{\gamma}} \int_{\Omega}\left|\frac{u}{t}+e\right|^{\widetilde{\gamma}} d x+\left\|a_{4}\right\|_{L^1(\Omega)} \rightarrow-\infty
\end{align*}
as $t \rightarrow+\infty$. Hence, there exists $r_{-}>r_{+}$ such that $\mathcal{I}(v) \leq 0$ for any $v \in \mathcal{C}_{k}^{-}+\left(\mathbb{R}^{+} e\right)$ with $\|v\|_{\alpha, \mu, p} \geq r_{-}$.

Now, for $r_{-}, r_{+}>0$, we set
\begin{align*}
&D_{-}=\left\{u \in \mathcal{C}_{k}^{-}:\|u\|_{\alpha, \mu, p} \leq r_{-}\right\},  \\
& S_{+}=\left\{u \in \mathcal{C}_{k}^{+}:\|u\|_{\alpha, \mu, p}=r_{+}\right\},\\
& Q=\left\{u+t e: u \in \mathcal{C}_{k}^{-}, t \geq 0,\|u+t e\|_{\alpha, \mu, p} \leq r_{-}\right\}, \\
& H=\left\{u+t e: u \in \mathcal{C}_{k}^{-}, t \geq 0,\|u+t e\|_{\alpha, \mu, p}=r_{-}\right\}.
\end{align*}

Since $\lambda_k<\lambda_{k+1}$, it follows from Theorem \ref{link index 2} that
\begin{equation*}
    i\left(\mathcal{C}_{k}^{-} \backslash\{0\}\right)=i\left(\mathcal{W}_{\alpha, \mu, p}(\Omega) \backslash \mathcal{C}_{k}^{+}\right)=k.
\end{equation*}
Moreover, we know that $\mathcal{C}_{k}^{-},$ $\mathcal{C}_{k}^{+}$ are two symmetric closed cones in $\mathcal{W}_{\alpha, \mu, p}(\Omega)$ with $\mathcal{C}_{k}^{-}\cap \mathcal{C}_{k}^{+}=\{0\}$. Then, applying Corollary \ref{link index} we conclude that ($Q, D_{-} \cup H$) links $S_{+}$ cohomologically in dimension $k+1$ over $\mathbb{Z}_{2}$. In particular, Proposition \ref{link index 3} ensures that $\left(Q, D_{-} \cup H\right)$ links $S_{+}$. 

Furthermore, by the geometric properties of a functional $\mathcal{I}$ stated above, we have that $\mathcal{I}$ is bounded on $Q$, $\mathcal{I}(u) \leq 0$ for every $u \in D_{-} \cup H$ and $\mathcal{I}(u) \geq \theta>0$ for every $u \in S_{+}$. Therefore, we have
$$
\sup _{D_{-} \cup H} \mathcal{I}<\inf _{S_{+}} \mathcal{I}, \quad \sup _Q \mathcal{I}<+\infty .
$$
In addition, by Proposition \ref{PS for linking}, the $(P S)_{c}$ condition holds. Then an application of Theorem \ref{link main thm} with 
$S=D_{-} \cup H,$ $D=Q,$ $A=S_{+}$ and $B=\emptyset,$ ensures that $\mathcal{I}$ admits a critical value $c \geq \theta$. Hence, there exists a critical point $u$ such that $\mathcal{I}(u)=c>0$. This implies the existence of a nontrivial weak solution to problem \eqref{main problem 1}.
\end{proof}

\subsection{Existence of a mountain pass solution}

Throughout this subsection, we assume that $\lambda<\lambda_1=1$.

We define the functional $\mathcal{J}: \mathcal{W}_{\alpha, \mu, p}(\Omega) \rightarrow \mathbb{R}$ as 
\begin{align}\label{functional MP or}
\begin{split}
    \mathcal{J}(u)&:=\frac{\alpha}{p} \int_{\Omega}|\nabla u|^{p} d x+\int_{(0,1)} \frac{c_{N, s, p}}{2p} \iint_{\mathcal{Q}} \frac{|u(x)-u(y)|^{p}}{|x-y|^{N+p s}} d x d y d \mu(s)\\
&+\frac{1}{p} \int_{\Omega}|u|^{p} d x-\frac{\lambda}{p} \int_{\Omega}|u|^{p} d x-\int_{\Omega} F(x, u) d x .
\end{split}
\end{align}
Note that $\mathcal{J}$ is of class $C^{1}$ and for all $u, v \in \mathcal{W}_{\alpha, \mu, p}(\Omega)$ we have 
\begin{align}\label{functional MP deriv or}
\begin{split}
    \left\langle \mathcal{J}^{\prime}(u), v\right\rangle= & \alpha \int_{\Omega} |\nabla u|^{p-2}\nabla u \nabla v d x\\
    &+\int_{(0,1)} \frac{c_{N, s, p}}{2} \iint_{\mathcal{Q}} \frac{\mathcal{A}(u(x)-u(y))(v(x)-v(y))}{|x-y|^{N+p s}} d x d y d \mu(s) \\
&+\int_{\Omega} |u|^{p-2}u v d x-\lambda \int_{\Omega} |u|^{p-2}u v d x-\int_{\Omega} f(x, u) v d x .
\end{split}
\end{align}
Moreover, the critical points of $\mathcal{J}$ are weak solutions of problem \eqref{main problem 1}.

\begin{prop}\label{MP PS prop}
Let $\lambda<\lambda_1$. Assume that $f$ satisfies \ref{Word:f1}, \ref{Word:f2} and \ref{Word:f3}. Then the functional $\mathcal{J}$, defined in \eqref{functional MP or}, satisfies the $(P S)_{c}$ condition for every $c \in \mathbb{R}$.
\end{prop} 

\begin{proof}  
Let $c \in \mathbb{R}$ and let $\{u_{n}\}\subset \mathcal{W}_{\alpha, \mu, p}(\Omega)$ be a sequence such that
\begin{equation}\label{ps mp c1 or}
\lim _{n \rightarrow+\infty} \mathcal{J}\left(u_{n}\right)=c 
\end{equation}
and
\begin{equation}\label{ps mp c2 or}
\lim _{n \rightarrow+\infty} \sup _{\substack{v \in \mathcal{W}_{\alpha, \mu, p}(\Omega) \\\|v\|_{\alpha, \mu, p}=1}}\left|\left\langle \mathcal{J}^{\prime}\left(u_{n}\right), v\right\rangle\right|=0. 
\end{equation}
Observe that we have
\begin{align}\label{functional MP equiv or}
\mathcal{J}(u_n) & =\frac{1}{ p}\|u_n\|_{\alpha, \mu, p}^{p}-\frac{\lambda}{p} \int_{\Omega}|u_n|^{p} d x-\int_{\Omega} F(x, u_n) d x.
\end{align}
Taking $\gamma$ as in \ref{Word:f3}, we have by \eqref{ps mp c1 or} and \eqref{ps mp c2 or} that
\begin{equation}\label{ps fpr mp tech 1 or}
\gamma \mathcal{J}\left(u_{n}\right)-\left\langle \mathcal{J}^{\prime}\left(u_{n}\right), u_{n}\right\rangle \leq C_1+C_2\|u_{n}\|_{\alpha,\mu,p} 
\end{equation}
for some constants $C_1, C_2>0$ and all $n \in \mathbb{N}$. On the other hand, by \eqref{functional MP deriv or} and \eqref{functional MP equiv or}, we have
\begin{align}\label{ps for mp tech 7 or}
\gamma \mathcal{J}\left(u_{n}\right)-\left\langle \mathcal{J}^{\prime}\left(u_{n}\right), u_{n}\right\rangle 
 &=\gamma\left(\frac{1}{p}\|u_n\|_{\alpha, \mu, p}^{p}-\frac{\lambda}{p}\int_{\Omega}|u_n|^pdx-\int_{\Omega} F(x,u_n)dx\right) \nonumber \\
&-\left(\|u_n\|_{\alpha, \mu, p}^{p}-\lambda \int_{\Omega}|u_n|^pdx-\int_{\Omega} f(x,u_n)u_ndx\right)  \\
&=\left(\frac{\gamma}{p}-1 \right)\|u_n\|_{\alpha, \mu, p}^{p}-\left(\lambda\left(\frac{\gamma}{p}-1\right)\right)\int_{\Omega}|u_n|^pdx \nonumber\\
&+\int_{\Omega}(f(x, u_n)u_n -\gamma F(x, u_n))dx  \nonumber.
\end{align}
Consider the term
\begin{align}\label{ps for mp tech 6 or}
\begin{split}
    \int_{\Omega}\left(f\left(x, u_n\right) u_n-\gamma F\left(x, u_n\right)\right) d x&=\int_{\Omega \cap{\left\{\left|u_n\right| > R\right\}}}\left(f\left(x, u_n\right) u_n-\gamma F\left(x, u_n\right)\right) d x\\
    &+\int_{\Omega \cap{\left\{\left|u_n\right| \leq R\right\}}}\left(f\left(x, u_n\right) u_n-\gamma F\left(x, u_n\right)\right) d x,
\end{split}
\end{align}
where $R\geq0$ is the constant introduced in \ref{Word:f3}.
As already proved in \eqref{ps for mp tech 4} and \eqref{ps for mp tech 5}, we have
\begin{equation}\label{ps for mp tech 4 or}
    \int_{\Omega \cap{\left\{\left|u_n\right| > R\right\}}}\left(f\left(x, u_n\right) u_n-\gamma F\left(x, u_n\right)\right) d x \geq 0
\end{equation}
and
\begin{equation}\label{ps for mp tech 5 or}
    \left|\int_{\Omega \cap{\left\{\left|u_n\right| \leq R\right\}}}\left(f\left(x, u_n\right) u_n-\gamma F\left(x, u_n\right)\right) d x \right| \leq C_R|\Omega|=\widetilde{C}_R
\end{equation}
for some $\widetilde{C}_R \geq 0.$
Combining \eqref{ps for mp tech 4 or} and \eqref{ps for mp tech 5 or} in \eqref{ps for mp tech 6 or}, and substituting the result into \eqref{ps for mp tech 7 or}, we obtain that
\begin{align}\label{ps for mp tech 8 or}
 \gamma \mathcal{J}\left(u_{n}\right)-\left\langle \mathcal{J}^{\prime}\left(u_{n}\right), u_{n}\right\rangle
 &\geq \left(\frac{\gamma}{p}-1 \right)\|u_n\|_{\alpha, \mu, p}^{p}-\left(\lambda\left(\frac{\gamma}{p}-1\right)\right)\int_{\Omega}|u_n|^pdx-\widetilde{C}_R \nonumber\\
 &\geq  \lambda^* \left(\frac{\gamma}{p}-1 \right)\|u_n\|_{\alpha, \mu, p}^{p}-\widetilde{C}_R, 
\end{align}
where we set $\lambda^*:=\min\{1,1-\lambda\}>0$.

In \eqref{ps for mp tech 8 or} we used the following inequalities
    \begin{align*} 
    \left(\frac{\gamma}{p}-1 \right)\|u_n\|_{\alpha, \mu, p}^{p}&-\left(\lambda\left(\frac{\gamma}{p}-1\right)\right)\int_{\Omega}|u_n|^pdx-\widetilde{C}_R\\
    &\geq\left(\frac{\gamma}{p}-1 \right)\min\{1,1-\lambda\}\|u_n\|_{\alpha, \mu, p}^{p}-\widetilde{C}_R\\
    &=\left\{\begin{array}{l}
 \left(\frac{\gamma}{p}-1 \right)\|u_n\|_{\alpha, \mu, p}^{p}-\widetilde{C}_R \text { if } \lambda \leq 0,  \\
(1-\lambda)\left(\frac{\gamma}{p}-1 \right)\|u_n\|_{\alpha, \mu, p}^{p}-\widetilde{C}_R \text{ if } 0<\lambda<1. 
\end{array}\right.
\end{align*}

Thus, by \eqref{ps fpr mp tech 1 or} and \eqref{ps for mp tech 8 or}, since $\gamma>p$, we deduce that the sequence $\{u_{n}\}$ is bounded in $\mathcal{W}_{\alpha, \mu, p}(\Omega)$. Hence, there exists $u\in \mathcal{W}_{\alpha, \mu, p}(\Omega)$ such that, as $n \rightarrow \infty$
\begin{equation}\label{weak conv 1 or}
    u_{n} \rightharpoonup u \text{ in } \mathcal{W}_{\alpha, \mu, p}(\Omega)
\end{equation}
and 
\begin{equation}\label{strong lp conv 1 or}
    u_{n} \rightarrow u \text{ in } L^{p}(\Omega),
\end{equation}
up to a subsequence. 

By \eqref{ps mp c2 or} and \eqref{weak conv 1 or}, we have that
\begin{equation*}
    \left\langle \mathcal{J}^{\prime}\left(u_{n}\right), u_{n}-u\right\rangle \rightarrow 0.
\end{equation*}
From this, arguing exactly as in the proof of Proposition \ref{PS for linking}, we conclude that $u_{n} \rightarrow u$ in $\mathcal{W}_{\alpha, \mu, p}(\Omega)$, verifying the $(PS)_c$ condition.
\end{proof}

\begin{prop}\label{MP geometry prop}
    Let $\lambda<\lambda_{1}$. Assume that $f$ satisfies assumptions \ref{Word:f1}, \ref{Word:f2} and \ref{Word:f4}. Then, there exist $\rho>0$ and $\beta>0$ such that, 

$(i)$ for any $u \in \mathcal{W}_{\alpha, \mu, p}(\Omega)$ with $\|u\|_{\alpha, \mu, p}=\rho$, it holds that $\mathcal{J}(u) \geq \beta$;

$(ii)$ there exists $e \in \mathcal{W}_{\alpha, \mu, p}(\Omega)$ such that $e \geq 0$ a.e. in $\mathbb{R}^{N},$ $\|e\|_{\alpha, \mu, p}>\rho$ and $\mathcal{J}(e)<\beta$.
\end{prop} 
\begin{proof}
    $(i)$ Let $u \in \mathcal{W}_{\alpha, \mu, p}(\Omega)$. By \eqref{estim for f ex}, \eqref{functional MP equiv or}, and Corollary \ref{embed full}, we have for any $\varepsilon>0$,
\begin{align}\label{mp geom tech 1}
\mathcal{J}(u) &=\frac{1}{p}\|u\|_{\alpha, \mu, p}^{p}-\frac{\lambda}{p}\|u\|_{L^{p}(\Omega)}^{p}-\int_{\Omega}F(x,u)dx \nonumber\\
&\geq \frac{1}{p}\|u\|_{\alpha, \mu, p}^{p}-\frac{\lambda}{p}\|u\|_{L^{p}(\Omega)}^{p}-\frac{\varepsilon}{p}\|u\|_{L^{p}(\Omega)}^{p}-\delta(\varepsilon)\|u\|_{L^{q}(\Omega)}^{q} \nonumber\\
& \geq \left(\frac{\lambda^*-\varepsilon}{p} \right)\|u\|_{\alpha, \mu, p}^{p}-\delta(\varepsilon)\|u\|_{L^q(\Omega)}^q  \\
& \geq \|u\|_{\alpha, \mu, p}^{p}\left(\frac{\lambda^*-\varepsilon}{p}-\delta(\varepsilon) c\|u\|_{\alpha, \mu, p}^{q-p}\right),\nonumber
\end{align}
where $\lambda^*:=\min\{1,1-\lambda\}>0$ and $c=c\left(N, \Omega, s_{\sharp}, p\right)>0$ denotes the embedding constant.

To establish \eqref{mp geom tech 1}, we used the following inequalities
     \begin{align*}
   \frac{1}{p}\|u\|_{\alpha, \mu, p}^{p}-\frac{\lambda}{p}\|u\|_{L^{p}(\Omega)}^{p} -\frac{\varepsilon}{p}\|u\|_{L^{p}(\Omega)}^{p}&-\delta(\varepsilon)\|u\|_{L^{q}(\Omega)}^{q} \geq \left(\frac{\lambda^*-\varepsilon}{p} \right)\|u\|_{\alpha, \mu, p}^{p}-\delta(\varepsilon)\|u\|_{L^q(\Omega)}^q\\
   &=\left\{\begin{array}{l}
 \left(\frac{1-\varepsilon}{p} \right)\|u\|_{\alpha, \mu, p}^{p}-\delta(\varepsilon)\|u\|_{L^q(\Omega)}^q \text { if } \lambda \leq 0,  \\
\left(\frac{1}{p}-\frac{\lambda}{p}-\frac{\varepsilon}{p} \right)\|u\|_{\alpha, \mu, p}^{p}-\delta(\varepsilon)\|u\|_{L^q(\Omega)}^q \text{ if } 0<\lambda<1. 
\end{array}\right.
\end{align*}
Since $\lambda^*>0$ and $q>p$, we can take
\begin{equation*}
    \varepsilon \in\left(0, \lambda^*\right) \quad \text{ and } \quad \rho \in\left(0,\left(\frac{\lambda^*- \varepsilon}{p \delta(\varepsilon) c}\right)^{\frac{1}{q-p}}\right)
\end{equation*}
to conclude that
\begin{equation*}
    \frac{\lambda^*-\varepsilon}{p}-\delta(\varepsilon) c \rho^{q-p}>0.
\end{equation*}
Consequently, taking $\|u\|_{\alpha, \mu, p}:=\rho$ in \eqref{mp geom tech 1}, we obtain
\begin{equation*}
    \inf _{\substack{u \in \mathcal{W}_{\alpha, \mu, p}(\Omega) \\\|u\|_{\alpha, \mu, p}=\rho}} \mathcal{J}(u) \geq \rho^{p}\left(\frac{\lambda^*-\varepsilon}{p}-\delta(\varepsilon) c \rho^{q-p}\right)=: \beta>0.
\end{equation*}

$(ii)$ Now we fix $u_{0} \in \mathcal{W}_{\alpha, \mu, p}(\Omega)$ such that $\left\|u_{0}\right\|_{\alpha, \mu, p}=1$ and $u_{0} \geq 0$ a.e. in $\mathbb{R}^{N}$. 
Then, for $t>0$, using \eqref{functional MP equiv or} and \ref{Word:f4}, we infer that
\begin{align}\label{mp geom tech 2}
\begin{split}
    \mathcal{J}\left(t u_{0}\right) & =\frac{t^{p}}{p}\left\|u_{0}\right\|_{\alpha, \mu, p}^{p}-\frac{\lambda t^{p}}{p}\left\|u_{0}\right\|_{L^{p}(\Omega)}^{p}-\int_{\Omega} F\left(x, t u_{0}\right) d x \\
& \leq \frac{t^{p}}{p}\lambda_*\left\|u_{0}\right\|_{\alpha, \mu, p}^{p}-a_{3} t^{\widetilde{\gamma}}\left\|u_{0}\right\|_{L^{\widetilde{\gamma}}(\Omega)}^{\widetilde{\gamma}}+\int_{\Omega} a_{4}(x) d x,
\end{split}
\end{align}
where $\lambda_*:=\max \{1,1-\lambda\}>0$. 

In \eqref{mp geom tech 2} we used the fact that
      \begin{align*}
    \left\|u_{0}\right\|_{\alpha, \mu, p}^{p}-\lambda\left\|u_{0}\right\|_{L^{p}(\Omega)}^{p}\leq \lambda_*\left\|u_{0}\right\|_{\alpha, \mu, p}^{p}=\left\{\begin{array}{l}
 \|u\|_{\alpha, \mu, p}^{p} \text { if } 0 \leq \lambda < 1,  \\
(1-\lambda)\|u\|_{\alpha, \mu, p}^{p} \text{ if } \lambda<0. 
\end{array}\right.
\end{align*}
Since $\widetilde{\gamma}>p$, passing to the limit in \eqref{mp geom tech 2}, we get
\begin{equation*}
    \lim _{t \rightarrow+\infty} \mathcal{J}\left(t u_{0}\right)=-\infty.
\end{equation*}
Therefore, taking $e:=t u_0$ for sufficiently large $t$, we obtain $\|e\|_{\alpha, \mu, p}>\rho$ and $\mathcal{J}(e)<\beta$, concluding the proof.
\end{proof} 

\begin{proof}[Proof of Theorem \ref{main result 03}]
For $\lambda<\lambda_1$, Proposition \ref{MP PS prop} guarantees that the Palais-Smale condition holds for every level $c \in \mathbb{R}$. Furthermore, Proposition \ref{MP geometry prop} establishes that the functional $\mathcal{J}$ satisfies the geometric assumptions of the mountain pass theorem. Consequently, applying Theorem \ref{MP thm}, we conclude that there exists a critical point $u \in \mathcal{W}_{\alpha, \mu, p}(\Omega)$ of $\mathcal{J}$ such that
\begin{equation*}
    \mathcal{J}(u) \geq \beta>0=\mathcal{J}(0).
\end{equation*}
In particular, we have $u \not \equiv 0$. This concludes the existence of a nontrivial weak solution to problem \eqref{main problem 1}.
\end{proof}

\section*{Conflict of interest statement}
On behalf of all authors, the corresponding author states that there is no conflict of interest.

\section*{Data availability statement}
Data sharing is not applicable to this article as no datasets were generated or analysed during the current study.

\section*{Acknowledgement}
 This work was supported by the FWO Odysseus 1 grant G.0H94.18N: Analysis and Partial Differential Equations and the Methusalem program of the Ghent University Special Research Fund (BOF) (Grant number 01M01021). The author is also supported by the Bolashak Government Scholarship of the Republic of Kazakhstan.

\end{document}